\documentclass[11pt]{article}
\usepackage{amsmath,amsfonts,amsthm}
\usepackage{amssymb}
\usepackage{fullpage}
\usepackage{enumerate}
\usepackage{graphicx}

\newtheorem{theorem}{Theorem}
\newtheorem{lemma}[theorem]{Lemma}
\newtheorem{corollary}[theorem]{Corollary}
\newtheorem{proposition}[theorem]{Proposition}

\newtheorem{problem}{Problem}
\newtheorem{conjecture}[theorem]{Conjecture}

\newcommand\RR{\ensuremath{\mathbb{R}}}
\newcommand\NN{\ensuremath{\mathbb{N}}}
\newcommand\ZZ{\ensuremath{\mathbb{Z}}}

\newcommand\diam{\mathop{diam}}
\newcommand\proj{\mathop{proj}}
\newcommand\Val{\mathop{{\hbox{\sc Val}}}}
\newcommand\ValC{\underline{\Val}_{\,C}}
\newcommand\ValR{\overline{\Val}_{R}}

\newcommand\T{\mathcal T}
\newcommand\VT{\T_0}
\newcommand\ET{\T_1}
\newcommand\ext{\mathop{ext}}
\renewcommand\int{\mathop{int}}
\newcommand\sys{\mathop{sys}}

\begin{document}

\title{The game of Cops and Robber on geodesic spaces}

\author{Bojan Mohar\thanks{Supported in part by the NSERC Discovery Grant R611450 (Canada),
and by the Research Project J1-2452 of ARRS (Slovenia).}%
~\thanks{On leave from IMFM, Department of Mathematics, University of Ljubljana.}\\
\small Department of Mathematics\\[-0.8ex]
\small Simon Fraser University\\[-0.8ex]
\small Burnaby, BC \ V5A 1S6, Canada\\
\small\tt mohar@sfu.ca}

\date{}

\maketitle

\begin{abstract}
The game of Cops and Robber is traditionally played on a finite graph. The purpose of this paper is to introduce and analyse the game that is played on an arbitrary geodesic space (a compact, path-connected space endowed with intrinsic metric). It is shown that the game played on metric graphs is essentially the same as the discrete game played on abstract graphs and that for every compact geodesic surface there is an integer $c$ such that $c$ cops can win the game against one robber, and $c$ only depends on the genus $g$ of the surface. It is shown that $c=3$ for orientable surfaces of genus $0$ or $1$ and nonorientable surfaces of crosscap number $1$ or $2$ (with any number of boundary components) and that $c=O(g)$ and that $c=\Omega(\sqrt{g})$ when the genus $g$ is larger.
The main motivation for discussing this game is to view the cop number (the minimum number of cops needed to catch the robber) as a new geometric invariant describing how complex is the geodesic space.
\end{abstract}


\section{Introduction}

In this article we discuss the game of cops and robbers on geodesic spaces, with particular attention to metric graphs and metric 2-dimensional cell complexes endowed with intrinsic metric. Our version of the game has been introduced in \cite{Mo21}, where it was shown that there is a min-max theorem for the game in which the cops try to minimize (and the robber want to maximize) the infimum of the distances between the robber and the cops during the gameplay.  This version of the game is somewhat different from the game version discussed by Bollob\'as, Leader and Walters \cite{BoLeWa12}, but it preserves all the beauty and power of discrete Cops and Robber game played on graphs, it generalizes the classic game of Man and Lion (and other games of that type), and is still completely intuitive and natural to deal with.

The game of Cops and Robber on graphs shows remarkable connections with structural graph theory. Our main motivation to analyse the game on geodesic spaces is to introduce the cop number of geodesic spaces. This is a new geometric invariant and our preliminary results show that it may be of certain interest when discussing geometry of groups and manifolds.

\subsection{Pursuit-evasion games}

Pursuit-evasion games have a long history, especially in the setup of \emph{differential games} \cite{Is65,Ji15,Le94,Pa70,Pe93}. Differential games with more pursuers were introduced in the 1970s, see, e.g. \cite{HaBr74,Ch76,Ps76,LePa85} or a more recent paper \cite{FeIbAlSa20} and references therein. A more recent important application is design of robot movement in complicated environment, see e.g.~\cite{AHRW17}.

A general class of pursuit-evasion games has been devised in discrete setting, where the game is played on a finite graph. Nowakowski and Winkler \cite{NoWi83} and Quilliot \cite{Qui78} independently introduced the game of Cop and Robber with a robber being chased by a single cop. Aigner and Fromme \cite{AiFr84} extended the game to include more than one cop. For each graph $G$ and a positive integer $k$, the \emph{Cops and Robber game} on $G$, involves two players. The first player controls $k$ \emph{cops} placed at the vertices of the graph, and the second player controls the \emph{robber}, who is also positioned at some vertex. While the players alternately move to adjacent vertices (or stay at their current position), the cops want to catch the robber and the robber wants to prevent this ever to happen. The main question is how many cops are needed on the given graph $G$ in order that they can guarantee the capture. The minimum such number of cops is termed as the \emph{cop number} $c(G)$ of the graph.

The game of cops and robbers gained interest because of its ties with structural graph theory. Classes of graphs that can be embedded in a surface of bounded genus \cite{AiFr84} and those that exclude some fixed graph as a minor \cite{An84,An86} have bounded cop number. In particular, Aigner and Fromme \cite{AiFr84} proved that all graphs that can be embedded in the plane have cop number at most 3. We refer to the monograph by Bonato and Nowakowski \cite{BoNo11} for further details about the history of the game and for overview of the main results. Additionally, we refer to survey papers \cite{BoMo17,BaBo12} which give more details about relations of the game to topological graph theory and cover details about Meyniel's Conjecture, which is considered the most outstanding open problem in this area.


The famous Lion and Man problem that was proposed by Richard Rado in the late 1930s and discussed in Littlewood's Miscellany \cite{Li53,Li86} is a version of the game with one pursuer (the Lion) and one evader (the Man). The man and the lion are within a circular arena (unit disk in the plane), they run with equal maximum speed. It seems that in order to avoid the lion, the man would choose to run on the boundary of the disk. A simple argument then shows that the lion could always catch the man by staying on the segment joining the center of the disk with the point of the man and slowly approaching him. However, Besicovitch proved in 1952 (see \cite[pp.~114--117]{Li86}) that the man has a simple strategy, in which he will approach but never reach the boundary. That strategy enables him to avoid capture forever, no matter what the lion does.
We refer to \cite{BoLeWa12} for more details.

One can prove that two lions are enough to catch the man in a disk. A recent work by Abrahamsen et al.\ \cite{AHRW17,AHRW20} discusses the game with many lions versus one man in an arbitrary compact subset of the plane whose boundary consists of finitely many rectifiable simple closed curves and prove that three lions can always get their prey. There are many extensions of the Man and Lion game \cite{BoLeWa12}; the extension to ``birds catching a fly" in the unit ball in ${\mathbb R}^n$ \cite{Croft64} is just one such example.

There is also a more recent study of differential games with many pursuers in convex compact sets in Euclidean spaces by Ferrara et al. \cite{FeIbAlSa20} and a study of the game on 1-skeletons of regular polytopes \cite{AzKuHo17,AzKuKh16,AzKuKh19} by Azamov, Kuchkarov, and Kholboev.

Klein and Suri \cite{KlSu15} discussed pursuit-evasion on polyhedral surfaces. In their setting, players make alternate discrete steps, all of same length $t$, and with $t\to 0$. This way, they approximate the continuous setting of the game. They proved that $4g+4$ pursuers can always catch evader on a polyhedral surface of genus $g$, thus approximately matching a result of Schr\"oder \cite{Schr01} that $\tfrac{3}{2}g+3$ cops suffice in the cops and robber game on graphs of genus $g$.

The game of cops and robbers can be defined on any metric space.
It is far from obvious how such a game can be defined in order to be natural, resembling interesting examples and allowing for powerful mathematical tools. Subtleties of the various versions of the game are nicely outlined in an influential article by Bollob\'as, Leader, and Walters \cite{BoLeWa12}, who were the first to provide a general framework for such games. In this paper we take a slightly different approach -- following \cite{Mo21} -- which yields a common generalization of discrete type for all of the above-mentioned versions.

\subsection{Overview}

The purpose of this paper is to analyse the game of cops and robber that is played on an arbitrary geodesic space (a compact, path-connected space endowed with intrinsic metric). It is shown (see Theorem \ref{thm:metric graph}) that the game played on metric graphs is essentially the same as the discrete game played on abstract graphs and that for every surface there is an integer $c$ such that $c$ cops can win the game against one robber, and $c$ only depends on the genus $g$ of the surface. It is shown that three cops (i.e. $c=3$) are sufficient to win the game on any compact geodesic surface $S$
if its genus is $0$ or $1$ (if $S$ is orientable) or its crosscap number is $1$ or $2$ (when $S$ is nonorientable) and with any number of boundary components. Genus 0 case is covered by our Theorem \ref{thm:genus 0} and genus 1 case is Theorem \ref{thm:genus 1}. The nonorientable surfaces are dealt easily through their orientable double cover. In general, when the Euler genus $g$ is larger, we may need  up to $\Omega(\sqrt{g})$ cops. On the other hand, we prove that $O(g)$ cops suffice.

\section{Intrinsic metric and geodesic spaces}

We consider a metric space $(X,d)$ and the corresponding metric space topology on $X$. For $x,y\in X$, an \emph{$(x,y)$-path} is a continuous map $\gamma: I\to X$ where $I=[0,1]$ is the unit interval on $\RR$ and $\gamma(0)=x$ and $\gamma(1)=y$.
We allow the paths to be parametrized differently and in particular we can replace $I$ with any finite interval on $\RR$.
The space is \emph{path-connected} if for any $x,y\in X$, there exists an $(x,y)$-path connecting them.

One can define the \emph{length} $\ell(\gamma)$ of the path $\gamma$ by taking the supremum over all finite sequences $0=t_0<t_1<t_2< \cdots < t_n=1$ of the values $\sum_{i=1}^n d(\gamma(t_{i-1}),\gamma(t_i))$. Note that $\ell(\gamma)$ may be infinite; if it is finite, we say that $\gamma$ is \emph{rectifiable}. Clearly, the length of any $(x,y)$-path is at least $d(x,y)$.  The metric space $X$ is a \emph{geodesic space} if for every $x,y\in X$ there is an $(x,y)$-path whose length is equal to $d(x,y)$.

An $(x,y)$-path $\gamma$ is \emph{isometric} if $\ell(\gamma) = d(x,y)$. Observe that for $0\le t < t' \le 1$ the subpath $\gamma|_{[t,t']}$ is also isometric. Therefore the set $\gamma(I) = \{\gamma(t)\mid t\in I\}$ is an isometric subset of $X$. With a slight abuse of terminology, we say that the image $\gamma(I)\subset X$ is an \emph{isometric path} in $X$.

A path $\gamma$ is a \emph{geodesic} if it is locally isometric, i.e., for every $t\in [0,1]$ there is an $\varepsilon>0$ such that the subpath $\gamma|_J$ on the interval $J = [t-\varepsilon,t+\varepsilon]\cap[0,1]$ is isometric. A path with $\gamma(0)=\gamma(1)$ is called a \emph{loop} (or a \emph{closed path}). When we say that a loop is a geodesic, we mean it is geodesic as a path and it is also locally isometric around its base point, i.e. $\gamma|_{[1-\varepsilon,1]\cup[0,\varepsilon]}$ is isometric for some $\varepsilon>0$.
We will mainly deal with \emph{$(x,y)$-geodesics}, which we define as shortest geodesics from $x$ to $y$, and thus we implicitly assume that any $(x,y)$-geodesic is always isometric.

One can consider any path-connected compact metric space $X$ and then define the shortest-path distance. For $x,y\in X$, the \emph{shortest-path distance} from $x$ to $y$ is defined as the infimum of the lengths of all $(x,y)$-paths in $X$. If any two points in $X$ are joined by a path of finite length, then the shortest path distance gives the same topology on $X$. Compactness implies that any sequence of $(x,y)$-paths of bounded length contains a point-wise convergent subsequence, and that the limit points determine an $(x,y)$-path. This implies that there is a path whose length is equal to the infimum of all path lengths. Hence, for this metric, which is also known as the \emph{intrinsic metric}, $X$ is a geodesic space.

If $X$ is a geodesic space, each of its points appears on a geodesic. But some points only appear as the end-points of isometric paths in $X$ and cannot appear as interior points of those. Such points will be referred to as \emph{corners}. All other points appear as internal points on geodesics in $X$ and are said to be \emph{regular}. It is obvious that regular points are dense in $X$. On the other hand, the set of corners can also be very rich. It may contain the whole boundary component, but in the interior of $X$, it is totally path-disconnected in the sense that every path containing only corners is either trivial (a single point), or is contained in $\partial X$. Still, the set of corners in the interior of $X$ can be uncountable (for example, it can contain the Cantor set).

We refer to \cite{BuBuIv01} and \cite{BuSh04} for further details on metric geometry, and on geodesic spaces in particular.
From now on we will assume that $X$ is a compact path-connected space, endowed with intrinsic metric; in other words, $X$ is a compact geodesic space.

\subsection{Metric cell complexes}

If a metric space $X$ is homeomorphic to a 1-dimensional cell complex (graph), then we say that $X$ is a \emph{metric graph}. 
If $G$ is an abstract graph and $w:E(G)\to \RR_+$ is a function specifying the length of each edge, we define the \emph{metric graph} $X(G,w)$ corresponding to $G$ and $w$ as the metric graph $G$ in which each edge $e$ is represented by a real interval of length $w(e)$. We write $X(G)$ if all edge-weights are equal to 1.

Metric graphs are just finite 1-dimensional complexes, endowed with intrinsic metric.
We can consider intrinsic metric in any finite cell complex. Here we usually assume that each cell $C$ is endowed with intrinsic metric inherited from some Euclidean space in which $C$ is embedded. Such a geodesic space is said to be a \emph{piecewise-linear geodesic space}. More generally, each cell may be a more complicated geodesic space, and then we consider the induced intrinsic metric of the cell complex.

In the special case of simplicial complexes we assume that the simplices are linearly embedded (as the convex hull of their vertices) in some ${\mathbb R}^n$, unless stated otherwise.

\subsection{Polyhedral surfaces}

A \emph{polyhedral surface} is a surface obtained from a set of disjoint polygons in the Euclidean plane $\RR^2$ by pairwise identifying their sides (which have to be of the same length for each identified pair). More precisely, let $D_1,\dots, D_m$ be a family of pairwise disjoint polygons, where $D_i$ has $k_i$ sides $A_{ij}$ ($1\le j\le k_i$). Let us take the set $S$ of the $k_1+\cdots + k_m$ sides of these polygons, $S = \{ A_{ij} \mid i\in [m], j\in [k_i] \}$, and consider an involution $\mu: S\to S$, called a \emph{matching} of these sides, such that any matched pair of sides have the same length: $\ell( \mu(A_{ij})) = \ell(A_{ij})$. For each $A_{ij}$, whose matching side $\mu(A_{ij})$ is not equal to $A_{ij}$, select an orientation and then identify the two sides by pasting them together according to the chosen orientations. It is easy to see that this way we always obtain a metric surface when we consider its intrinsic metric. It is clear that every geodesic path in a polyhedral surface consists of straight-line segments inside the polygons $D_1,\dots,D_m$. The edges of the polygons that are matched with themselves form the boundary of the resulting surface. We assume, though, that the polyhedral surface is connected. If all polygons forming a polyhedral surface are triangles, then we say that the surface is \emph{triangulated}.

The polygons $D_1,\dots,D_m$ are said to be the \emph{faces} (or sometimes \emph{$2$-faces}) of the polyhedral surface, each side $A_{ij}$ identified with $\mu(A_{ij})$ is an \emph{edge} (or a \emph{$1$-face}), the endpoints of the edges are the \emph{vertices} (or \emph{$0$-faces}). Each vertex has several polygons identified cyclically around it and may be part of the boundary or an interior vertex. The corners of the polyhedral surface are precisely those vertices, whose incident polygons make the total angle around the vertex smaller than $2\pi$. We refer to \cite{AlZa67} for further treatment of polyhedral surfaces.

\subsection{Geodesic triangulations of metric surfaces}

If a geodesic space $X$ is homeomorphic to a surface, we say that $X$ is a \emph{metric surface}.

Alexandrov and Zalgaller \cite{Al48,Za56} have shown (with full proofs in their monograph \cite{AlZa67}) that every metric surface with bounded curvature can be partitioned into convex triangles with disjoint interiors that form a triangulation of the surface. With that result in hand, they also proved that every such surface can be approximated by a polyhedral surface and can also be approximated by a surface with Riemannian local geometry. Quite recently, these results have been extended to arbitrary geodesic surfaces with one caveat that the dissection of the surface into triangles need not form a triangulation in the sense that adjacent triangles need not intersect along entire edges (see Creutz and Romney \cite{CrRo21}).

A closed subset $T$ of $X$ is a \emph{triangle} if it is homeomorphic to a disk and its boundary can be written as the union of three isometric paths. It is possible that the boundary of a triangle can also be written as the union of two isometric paths, in which case it is also called a \emph{digon}. A triangle $T$ is \emph{convex} if for any $x,y\in T$, some $(x,y)$-geodesic is contained in $T$. It is \emph{non-degenerate} if the lengths of its sides satisfy the strict triangular inequalities.
We say that a family of subsets of $X$ is \emph{non-overlapping} if any two subsets in the family have disjoint interiors, and that it is \emph{locally finite} if every point of $X$ is contained only in finitely many of these subsets.

\begin{theorem}[\cite{CrRo21}]\label{thm:triangulate surface epsilon}
  Let $X$ be a geodesic surface endowed with intrinsic metric such that every boundary component is a piecewise geodesic curve, and let $\varepsilon>0$. Then $X$ is the union of a locally finite collection of non-overlapping triangles $(T_i)_{i\in J}$, and each triangle $T_i$ $(i\in J)$ has the following properties:
  \begin{enumerate}[\rm (i)]
    \setlength\itemsep{0pt}
    \item $T_i$ is convex,
    \item $\diam(T_i) \le \varepsilon$,
    \item $T_i$ is non-degenerate, and
    \item if a point $p\in \partial T_i$ is a corner, then $p\in\partial X$.
  \end{enumerate}
\end{theorem}

We will need an extension of this result to deal with boundary components that are not piecewise geodesic. See Theorem \ref{thm:make boundary piecewise geodesic} in the appendix.

\subsection{Topological and metric properties of geodesic surfaces}

Topologically, every compact surface is determined by three parameters: \emph{orientability} (either being orientable or nonorientable), its \emph{genus}, and the \emph{number of boundary components}. The orientable surface $S=S_{g,k}$ of genus $g$ with $k$ boundary components ($g\ge0, k\ge0$) has \emph{Euler characteristic} $\chi(S)=2-2g-k$ and the nonorientable surface $N=N_{g,k}$ of genus $g$ with $k$ boundary components ($g\ge1, k\ge0$) has \emph{Euler characteristic} $\chi(N)=2-g-k$. In order to treat surfaces with the same Euler characteristic alike, we say that $S_{g,k}$ and $N_{g,k}$ have \emph{Euler genus} equal to $2g$ and $g$, respectively.

However, when it comes to the intrinsic metric on geodesic surfaces, there are other parameters that distinguish them. We list some that we will use in the paper. Let $X$ be a geodesic surface. Then we define:
\begin{itemize}
  \item The \emph{diameter} of $X$ is the maximum distance between points in $X$, $$\diam(X) = \max \{d(x,y)\mid x,y\in X\}.$$
  \item \emph{Boundary separation} is the minimum distance between two distinct boundary components. We define the boundary separation be infinite if there is at most one boundary component.
  \item \emph{Systolic girth}, denoted as $\sys(X)$, is the length of a shortest noncontractible curve on the surface, and such a curve is said to be a \emph{systole} in $X$. If the surface is simply connected ($X = S_{0,0}$ or $S_{0,1}$), then the systolic girth is infinite.
  \item \emph{Essential systolic girth}, denoted as $\sys_0(X)$, is the length of a shortest curve in $X$ that is noncontractible on the surface $\widehat X$ that is obtained from $X$ by capping off each boundary component of $X$.
  \item For a boundary component $B\subseteq \partial X$, the \emph{systolic girth for $B$}, $\sys(X,B)$, is the minimum length of a noncontractible curve with its endpoints on $B$ fixed. Here we say the curve with its endpoints $x,y\in B$ is \emph{contractible} if it can be deformed to one of the two $(x,y)$-segments on $B$, where the contraction homotopy keeps the two endpoints $x$ and $y$ fixed all the time.
\end{itemize}

It is easy to see that systole curves are geodesic in $X$ and that each systole and essential systole is a simple closed curve.

\section{Game of Cops and Robber on geodesic spaces}

\subsection*{Rules of the game}

A more general setup for the game has been introduced in \cite{Mo21}. Here we only treat the \emph{standard game}.
Let $X$ be a compact, path-connected metric space endowed with intrinsic metric $d$, and let $k\ge1$ be an integer. The \emph{Game of Cops and Robber} on $X$ with $k$ cops is defined as follows. The first player, who controls the robber, selects the \emph{initial positions} for the robber and for each of the $k$ cops. Formally, this is a pair $(r^0,c^0)\in X^{k+1}$, where $r^0\in X$ is robber's position and $c^0 = (c_1^0,\dots,c_k^0)\in X^k$ are positions of the cops. The same player selects his \emph{agility function}, which is a map $\tau: \NN\to\RR_+$ and defines the lengths of the steps of the game.
The agility function must allow for the total duration of the game to be infinite, which means that $\sum_{n\ge1} \tau(n) = \infty$.

After the initial position and the agility function are chosen, the game proceeds as a discrete game in consecutive steps. Having made $n-1$ steps $(n\ge1)$, the players have their positions $(r^{n-1},c_1^{n-1},\dots,c_k^{n-1})\in X^{k+1}$. The $n$th step will have its duration determined by the agility: the move will last for time $\tau(n)$, and each player can move with unit speed up to a distance at most $\tau(n)$ from his current position.
So, the robber moves to a point $r^n\in X$ at distance at most $\tau(n)$ from its current position, i.e. $d(r^{n-1}, r^n)\le \tau(n)$. The destination $r^n$ is revealed to the cops. Then each cop $C_i$ ($i\in[k]$) selects his new position $c_i^n$ at distance at most $\tau(n)$ from its current position, i.e. $d(c_i^{n-1}, c_i^n)\le \tau(n)$. The game stops if $c_i^n = r^n$ for some $i\in[k]$. In that case, the \emph{value of the game} is 0 and we say that the cops \emph{have caught} the robber. Otherwise the game proceeds.
If it never stops, the \emph{value of the game} is
\begin{equation}\label{eq:value of game}
v = \inf_{n\ge0} \min_{i\in[k]} d(r^n, c_i^n).
\end{equation}
If the value is 0, we say that the \emph{cops won} the game; otherwise the \emph{robber wins}. Note that the cops can win even if they never catch the robber.

\subsection*{The cop number of a game space}

The compact geodesic space $X$ on which we play the game of Cops and Robber will be referred to as the \emph{game space}.
Given a game space $X$, let $k$ be the minimum integer such that $k$ cops win the game on $X$. This minimum value will be denoted by $c(X)$ and called the \emph{cop number} of $X$. If such a $k$ does not exist, then we set $c(X)=\infty$. Similarly we define the \emph{strong cop number} $c_0(X)$ as the minimum $k$ such that $k$ cops can always catch the robber.

Since $X$ is compact, for every $\varepsilon>0$ there exists an integer $k$ such that $k$ cops can always achieve the value of the game be less than $\varepsilon$. (Place the cops at the centers of open balls of radius $\varepsilon$ that cover $X$. Then, no matter where the robber is, he will be at distance less than $\varepsilon$ from one of the cops.) Hence, with the growing number of cops, the value of the game tends to 0. This simple observation does not guarantee that $c(X)$ is finite. However, we conjecture the following.

\begin{conjecture}
  Every game space $X$ has a finite cop number, $c(X) < \infty$.
\end{conjecture}

However, there are geodesic spaces, whose strong cop number $c_0(X)$ is infinite.  One example that came out in \cite{IrMo22} is a particular compact subspace of the unit ball in $\ell^2(\NN)$, which is defined as the set of all infinite sequences $(x_n)_{n\ge1}$ with $\ell^2$-norm $\sum_{n\ge1}x_n^2 \le1$. In this ball, one cop can use the \emph{radial strategy} of Besicovitch (see \cite{IrMo22}) to approach the robber arbitrarily close. However, no finite number of cops can catch the robber. The message of this example is that $c(X)$ says something about the global structure, but $c_0(X)$ bears some information about the local properties, which is about the dimension in this case.

\subsubsection*{Strategies and value of the game}

The value of the game when played in $X$ is defined by (\ref{eq:value of game}). Note that the strategy of a player depends not only on the current position but also on the agility chosen by the robber at the very beginning. To formalize this dependence, we define the \emph{full strategy} of the robber in such a way that it includes the choice of initial position and agility. This choice is a formal part of his strategy, which is assumed implicitly, and having made this choices, we then treat the game as starting at some initial position and having a fixed agility $\tau$. The rest of the robber's strategy and a strategy of the cops can be defined as follows. A \emph{strategy of the robber} is a function $s: X\times X^k\times \RR_+ \to X$, $(r,c,t) \mapsto r'$, such that $d(r,r')\le t$. This can be interpreted as having the robber and the cops in position $(r,c)$, and having step of duration $t$. The strategy tells us to move the robber from $r$ to $r'$ along some geodesic of length $d(r,r')$. Then, each cop $C_i$ moves from his current position $c_i$ to a point $c_i'$ at distance at most $t$ from $c_i$. The choice of such destinations $c' = (c_1',\dots,c_k')$ constitutes a \emph{strategy of cops}. Formally, it is a function $q: (r',c,t) \mapsto c'$. Performing the game by using both strategies gives the new position $(r',c')$.

Using agility $\tau$, initial position $(r^0,c^0)$ and strategies $s,q$ of the robber and the cops, we denote by $v_\tau(s,q)$ the value of the game when it is played using these strategies. Here we implicitly assume the initial position $(r^0,c^0)$ and agility $\tau$ selected by the robber are part of the strategy $s$. Now we define the \emph{guaranteed outcome} for each of the players. First for the robber:
$$
    \ValR(\tau) = \inf_q \sup_s v_\tau(s,q) \quad \textrm{and} \quad  \ValR = \sup_\tau \ValR(\tau).
$$
The inf-sup is considered for an arbitrary fixed agility $\tau$ and $q$ and $s$ run over all strategies of the cops and the robber (respectively). Similarly, the \emph{guaranteed outcome} for the cops is
$$
    \ValC(\tau) = \sup_s \inf_q v_\tau(s,q) \quad \textrm{and} \quad  \ValC = \sup_\tau \ValC(\tau).
$$
For each $\varepsilon>0$, there is $q$ such that for every $s$, $v_\tau(s,q) < \ValR(\tau)+\varepsilon$. This implies that $$\ValC(\tau) \le \ValR(\tau) \quad \textrm{and} \quad  \ValC \le \ValR.$$
If $\ValC = 0$, then we say that \emph{cops win} the game. If $\ValR>0$, then the \emph{robber wins}.

It is an interesting question whether it can happen that $\ValC < \ValR$ for some game space $X$ and some $k$. In particular, is it possible that both players, the cops and the robber win the game? This question was offered as the main open problem in the afore-mentioned work by Bollob\'as et al.~\cite{BoLeWa12}. For our version of the game, this cannot happen. Namely, the following ``min-max theorem" was proved in \cite{Mo21}. In that paper it is first shown that it is always in favor of the robber to use decreasing agility functions (meaning $\tau(n+1)\le \tau(n)$ for all $n\ge1$), in which case one can prove that the values of $\ValC$ and $\ValR$ are the same.

\begin{theorem}\label{thm:ValC=ValR}
  For every decreasing agility $\tau$ we have $\ValC(\tau)=\ValR(\tau)$. Consequently, $\ValC=\ValR$.
\end{theorem}

This ``min-max theorem" implies that for every $\varepsilon>0$, there are near-optimal strategies for both players, and if either one of them uses his strategy, the other player cannot do more than $\varepsilon$ better than just using his own near-optimal strategy. The precise statement is given next.

\begin{corollary}
Suppose that $\tau$ is a decreasing agility function and that $\varepsilon>0$. There are strategies $s_\varepsilon$ and $q_\varepsilon$ for the robber and the cops (respectively) such that for any strategy $s$ of the robber and any strategy $q$ of the cops,
$$
  v_\tau(s,q_\varepsilon) - \varepsilon < v_\tau(s_\varepsilon,q_\varepsilon) < v_\tau(s_\varepsilon,q) + \varepsilon.
$$
Consequently,
$$
  \ValR(\tau)-\varepsilon \le v_\tau(s_\varepsilon,q_\varepsilon)\le \ValC(\tau) + \varepsilon.
$$
\end{corollary}

This corollary gives an important outcome that any cops and robber game can be described as a finite game within an arbitrary precision. This approximation result is described next.

Suppose that we fix agility $\tau$ and a positive integer $N$. Then we can consider $N$ steps of the game, and let $T=T_N(\tau)=\sum_{i=1}^N \tau(i)$ be the duration of the game during these $N$ steps. We say that these $N$ steps present an \emph{$\varepsilon$-approaching game for agility $\tau$} if the cops have a strategy $q_\varepsilon$ such that within these $N$ steps, their distance from the robber is at most $\ValC(\tau) + \varepsilon$.

\begin{proposition}\label{prop:epsilon-game exists}
  For every agility $\tau$ and every $\varepsilon>0$, there is an $\varepsilon$-approaching game with finitely many steps.
\end{proposition}

After making $N$ steps of the $\varepsilon$-approaching game for agility $\tau$, using strategy $q_\varepsilon$ of the cops, the players come to certain position and then they continue playing. Now, the cops can use $\varepsilon/2$-approaching game for the remaining agility in steps $N+1,N+2,\dots$ using strategy $q_{\varepsilon/2}$. By definition of $\varepsilon/2$-approaching game, they come within distance $\varepsilon/2$ from the value of the game. Next, they can use $\varepsilon/3$-approaching game for the remaining agility, and so on. If the agility $\tau$ is decreasing, the strategies $q_{\varepsilon}, q_{\varepsilon/2}, q_{\varepsilon/3},\dots$ can be combined into a single strategy that is optimal for the game on $X$ since the cops come arbitrarily close to the value of the game.

Given an $\varepsilon>0$, we define $c_\varepsilon(X)$ as the minimum number of cops that guarantee to win the $\varepsilon$-approaching game on $X$. With this notation, we have the following consequence which we state for further reference.

\begin{corollary}\label{cor:epsilon-approaching suffices}
  $c(X) = \sup \{ c_\varepsilon(X) \mid \varepsilon>0 \}$.
\end{corollary}

\subsection*{Game with a tail}

When having one more cop does not matter, we may assume that one of the cops, called the \emph{tail}, will just follow the trajectory of the robber, possibly making some shortcuts. More precisely, the strategy of the tail will be at each step to find a shortest path from its position $T_0$ to the current position of the robber and start moving full speed along that geodesic. This will assure that the distance from the tail to the robber will never increase. It is clear that with this strategy of the tail, the robber cannot rest, he must almost all the time use almost full speed. We will make this fact more precise below.

Suppose that the robber has optimal strategy $s_0$ for the $\varepsilon$-approximated game and that $s_0$ includes the decision $s_0(r,c,t)=r'$, where $d(r,r')<t$. Then we say that the robber \emph{rested} for $t-d(r,r')$ time units. Suppose that $k$ cops play optimal strategy for the game with $k$ cops and that they add a tail to this strategy. Then they can come arbitrarily close to the robber. Once this is achieved, the cop that is closest to the robber becomes the tail and the other $k$ players start the $\varepsilon$-approximated game from the start (of course, the agility is now different since the game continues with the next step). If the strategy of the robber would allow for the total rest of more than time $\varepsilon$ during the remaining gameplay, the tail would catch him or would come closer to the robber than the value of the game allows. This would be a contradiction. As the conclusion we have the following statement.

\begin{theorem}\label{th:with tail never rest}
  Suppose that the robber wins the game against $k$ cops plus a tail on the game space $X$. Then the robber has $\varepsilon$-approximate strategy against $k$ cops in which he never rests.
\end{theorem}

\begin{proof}
  The robber can use the strategy where the total resting time is less than $\varepsilon/2$. This way his game value remains more than $\ValR-\varepsilon/2$. Now, he can also keep going full speed without resting and will not make more than $\varepsilon/2$ away from where he would end up playing the $(\varepsilon/2)$-approximate strategy with tiny bit of resting. So, he would be at most $\varepsilon$ away from $\ValR$.
\end{proof}

\section{Game on metric graphs}

In this section we show that the game of cops and robber on metric graphs is essentially equivalent to the combinatorial game played on abstract graphs.

Let $G$ be a graph, and let $X=X(G)$ be the metric graph obtained from $G$ by considering each edge to be homeomorphic to the interval of length $1$ with the metric induced from the real interval $[0,1]$. Define $c(G)$ and $c(X)$ as the cop number of the graph $G$ (as in \cite{BoNo11}) and the cop number for the metric graph $X$ as defined in this paper, respectively.

\begin{theorem}\label{thm:metric graph}
  $c(G) \le c(X) \le c_0(X) \le c(G)+1$.
\end{theorem}

\begin{proof}
Let us consider a winning strategy of the robber played on the graph $G$ against $k=c(G)-1$ cops. For the metric graph $X$, the robber chooses the same initial positions for himself and the cops at the vertices of $X$ and chooses his agility so that he can make move of length 1 at each step. He will follow his discrete strategy in $G$ and will be at a vertex at the completion of each move.

For each cop $C_i$, consider his position $c_i$ at the end of the current move. Let $v$ be the vertex of $G$ that is closest to $c_i$ in $X$. In the unlikely (but possible) event that $C_i$ is at distance $\tfrac{1}{2}$ from two vertices, we let $v$ be the last vertex visited by $C_i$ during the game. Since the cops start in vertices of $X$, this is well-defined. We call the vertex $v\in V(G)$ determined by this rule the \emph{shadow} of $C_i$ in $G$. Now, the robber considers his position and the positions of the shadows of the cops in $G$ and makes the move according to his strategy in $G$. After the cops move, their shadows move according to the rules of the game (either stay at the same position or move to an adjacent vertex). Since the strategy of the robber on $G$ is a winning strategy, the shadows never catch the robber on $G$. It is easy to see that this means that the cops do not catch the robber on $X$ and that the minimum distance from the robber to any of the cops is at least $\tfrac{1}{2}$ at all times. This proves that $c(X)\ge c(G)$.

To prove the upper bound, $k=c(G)$ cops (they will be called \emph{regular cops}) will mimic their optimal strategy from $G$, and one additional cop will serve as a tail. The complicated ingredient in this part of the proof is that the robber chooses an agility that is hard to mimic in the discrete game on the graph. The game play is divided into two stages. In the first stage, the regular cops move into the vertices of $X$ and stay there until the robber also enters a vertex. Note that this may happen in the middle of step $n$ whose total duration is $\tau(n)$, but the robber arrives to a vertex $v$ at time $t'<\tau(n)$, and proceeds by moving further to a point $x$ that is not a vertex. We may assume that the distance of $x$ from $v$ is $t$, where $0<t<1$. Now we consider stopping phase 1 when the robber is at $v$ and consider the rest of this move as a new step of duration $t$. But that step will follow the strategy from stage two, which we describe next.

In the second stage, the regular cops consider their positions in $G$ after moving for a total length 1 into an adjacent vertex (or staying at the same vertex) following their strategy in $G$. If we know where the robber is going to move to, the strategy tells the cop $C_i$ ($1\le i\le k$) to move from his current vertex $v_i$ to an adjacent vertex $u_i$ (possibly $u_i=v_i$). Their movement is devised in such a way that after they achieve this, the robber will also be at a vertex. This is achieved as follows.  The robber announces that he will start moving from $v$ using the edge $vu$ towards $u$. (If the robber announces that he will stay at $v$, the cops also stay put, except the tail approaches $v$, and thus we may assume the move towards $u$ is real.)  We may also assume that the robber does not go past the vertex during his move by splitting his move in the same way as we did in order to complete stage one. Thus the robber ends up at a point $x$ on the edge $vu$ at distance $t$ form $v$, where $0\le t\le 1$. According to their strategy (using $u$ as the intended move of the robber), each cop $C_i$ will mimic the movement of the robber along $vu$ on the edge $v_iu_i$, so that he ends up at the same distance $t$ from $v_i$. Sooner or later the robber reaches a vertex (or is caught by the tail). If the vertex is $v$, the cops are also back at their former positions, so all that has changed is that the tail is closer.  We now repeat the strategy and we may assume that the robber reached $u$. At the same time the cops reach their destinations $u_i$. This may be in the middle of the $n$th move, but the cops have known the rest of the move of the robber before the move, so they were able to plan their continuation according to their strategy in $G$.

Since the cops win in $G$, there is a time when they catch a robber in $G$, and then they also catch him in $X$ (unless the tail catches the robber before). This shows that $c(X)\le c(G)+1$.
\end{proof}

A discrete version of this theorem for the game on graphs, where each edge of a graph is replaced by a path of length $r$ ($r$ is the same for every edge), was proved by S.~A.~Hosseini in his PhD Thesis \cite[Lemmas 5.7 and 5.8]{SAHosseini_Thesis}.

Let us observe that both bounds of Theorem \ref{thm:metric graph} are tight. Examples where the upper bound is attained are provided by any cop-win graph which is not a tree.

When we allow for general edge-weights, metric graphs generalize the game of cops and robbers. However, it seems that the generalization does not go far from the graph case. For example, if all weights are rational numbers, then we may as well assume the weights are integral, and in that case, $X(G,w)=X(S_w(G))$, where $S_w(G)$ is a subdivision of $G$ in which each edge $e$ of $G$ is replaced by a path of length $w(e)$, and then the edges in $S_w(G)$ are considered of unit length. Nevertheless, metric graphs bring new problems that may be of interest. Let us describe some possible directions.

\subsection*{How many more cops may be needed for capture}

\begin{problem}\label{prb:c<c_0}
  Is there a metric graph $X=X(G,w)$ such that $c(X) < c_0(X)$? If strict inequality is possible, how large can $c_0(X)-c(X)$ be?
\end{problem}

It is possible that the answer to the second question in Problem \ref{prb:c<c_0} is that the difference $c_0(X)-c(X)$ is bounded above by a constant. However, having no answers to the first question, we are not able to make a good conjecture.

We have a partial result.

\begin{proposition}
    For every metric graph $X=X(G,w)$ with rational edge-lengths, we have $c_0(X) \le c(X)+1$.
\end{proposition}

\begin{proof}
  Let $S=S_w(G)$. Then Theorem \ref{thm:metric graph} shows that $k\le c(X)\le c_0(X)\le k+1$, where $k=c(S)$. This implies that $c_0(X) \le c(X)+1$.
\end{proof}

\subsection*{Meyniel conjecture for metric graphs}

One of the outstanding open problems about the cop number of graphs is the Meyniel Conjecture (see \cite{BaBo12}), which asserts that there is a constant $c$ such that the cop number of any $n$-vertex graph is at most $c\sqrt{n}$. Here we propose a generalization.

\begin{conjecture}\label{conj:metric Meyniel}
  The cop number of any metric graph $X(G,w)$ with $n=|V(G)|$ vertices is $O(\sqrt{n}\,)$. In other words, there is a constant $\alpha$ such that
  $$
     c_0(X(G,w)) \le  \alpha\sqrt{n}.
  $$
\end{conjecture}

\subsection*{Graph minors}

Let us consider a metric graph $X=X(G)$ or $X=X(G,w)$. If we let one edge-weight to be increased to a very large value, this has the same effect on the cop number as deleting this edge (if the edge is not a cutedge and there are at least 2 cops). On the other hand, lowering the weight of the edge to 0 is the same as contracting the edge. Thus all possible weightings describe somewhat more general set than all minors of the graph $G$.

We believe that considering the following question would be worthwhile.

\begin{problem}
  Let\/ $G$ be a graph. What is the maximum of $c(X(G,w))$ taken over all metric graphs based on the graph $G$?
\end{problem}

The set of all metric graphs $X(G,w)$ with all weights in $[0,1]$ is a polytope. Can we describe where the maximal values of $c(X(G,w))$ occur? In particular, the following is a problem that asks about the same as Conjecture \ref{conj:metric Meyniel}.

\begin{problem}
  What is the maximum of $c(X(K_n,w))$ taken over all metric graphs based on the complete graph $K_n$?
\end{problem}

\section{Tools}

The basic strategy of cops, the isometric path lemma of Aigner and Fromme \cite{AiFr84}, works in our setting as well.

\begin{lemma}\label{lem:isometric path}
Let $I$ be an isometric path in $X$. Then one cop can guard $I$. More precisely, after the cop reaches $I$ and spends time equal to the length of $I$ on the path to adjust himself, whenever the robber will steps on $I$ or cross it, he will be caught by the cop in the same step.
\end{lemma}

\begin{proof}
The proof is essentially the same as in the case of geodesic paths in graphs. Let $a,b$ be the ends of $I$, and let $L$ be the length of $I$. Then we define, for each point $r\in X$, its \emph{shadow} $\sigma(r)\in I$ as follows. If $d(r,a) \ge L$, then we set $\sigma(r)=b$. Otherwise, we let $\sigma(r)$ be the point on $I$ whose distance from $a$ is equal to $d(r,a)$. Initially, the cop moves to $a$ and then progresses towards $b$ until he reaches the shadow of the robber. From that point on, he stays at the shadow all the time. This can be maintained since for any $r,r'\in X$, $d(\sigma(r),\sigma(r')) \le d(r,r')$. This strategy works well in the continuous and in the discrete version. In our setting, if the robber moves from $r_0$ to $r$, and his path contains a point $r_1\in I$, then $d(r_0,r_1) \ge d(\sigma(r),r_1)$. Thus, after the move of the robber is complete, the cop can move from $\sigma(r_0)$ to $r_1$ on $I$ and then follow the path of the robber from $r_1$ to $r$, capturing him.
\end{proof}

This result has been generalized to graph retracts in the graph case. In the case of geodesic spaces, we can generalize it to the following notion adapted to metric geometry.

A map $\sigma: X\to X$ is a \emph{$1$-Lipschitz function} if for every $x,y\in X$,
$$
   d(\sigma(x),\sigma(y)) \le d(x,y).
$$
Having such a function, the image $\sigma(x)$ of $x$ will be referred to as the \emph{$\sigma$-shadow} of $x$ (or simply a \emph{shadow} of $x$ if $\sigma$ is clear from the context). Note that every $1$-Lipschitz function is continuous, thus $Y=\sigma(X)$ is a connected geodesic space. Suppose now that $k$ cops can catch the robber in $Y$.  A strategy to catch the robber in $Y$ can be used to \emph{catch the shadow} of the robber when playing the game on $X$. This means the following gameplay. Let $r\in X$ be the position of the robber. If there is a cop whose position is $\sigma(r)$, then there is nothing to do, as we already have a cop at the shadow of $r$. Having achieved this, the shadow has been caught. If not, then we first bring $k$ cops into $Y$ and then we consider the position $\sigma(r)$ as being a robber in $Y$ and use the cops to catch that robber. While $r$ is moving in $X$, $\sigma(r)$ is moving in $Y$. Since $\sigma$ is $1$-Lipschitz, any move of the shadow is consistent with the rules of the game in $Y$, and by our assumption, the cops can catch the shadow.

Once the cops catch the shadow of $r$ in $Y$, the cop at the shadow can follow every move of the robber and stay in the shadow of $r$ indefinitely. Let us now consider the set of fixed points of $\sigma$:
$$
    \Phi(\sigma) := \{ x\in X \mid \sigma(x)=x \}.
$$
If a cop follows the shadow, then whenever the robber passes through a point in $\Phi(\sigma)$, the cop in the shadow will catch the robber. (This is true also if the robber just passes through such point in the interior point of a geodesic defining his current move.) Thus we say that the cop following the shadow can \emph{guard} the set $\Phi(\sigma)$.

For the above setup of catching the shadow, more than one cop may be used, but once the shadow has been caught (or approached within distance $\varepsilon$), just one cop can remain in the shadow (or stay within distance $\varepsilon$), the rest of them can be released.

Note that guarding an isometric path (Lemma \ref{lem:isometric path}) is a special case of guarding the fixed-point set of a $1$-Lipschitz function.

Let us look at some additional examples of 1-Lipschitz functions that can be used to guard a subset $Y$ of a game space.

(a) Suppose that $X\subset \RR^n$ and that $H$ is a hyperplane (of any codimension) in $\RR^n$. Let $\sigma(x)=\proj_H(x)$ be the orthogonal projection onto $H$. If for each $x\in X$, the projection $\sigma(x)$ is also in $X$, then $\sigma$ is 1-Lipschitz and $\Phi(\sigma)=X\cap H$. This shows that we can guard the intersection of the hyperplane with $X$ with one cop as long as we can catch the shadow in $X\cap H$.

(b) More generally, suppose that $X\subseteq Y\times Z$ with $Y_0 := Y\times \{z_0\} \subseteq X$ for some $z_0$ and that the metric in the product satisfies: $d_X((y,z),(y',z')) \ge d_Y(y,y')$ and that $Y_0$ is isometric with $Y$. Then the projection to $Y_0$ is 1-Lipschitz and one cop can guard $Y_0$.

(c) Any isometry $\sigma:X\to X$ is 1-Lipschitz. So we can guard the fixed-point-set of the isometry as long as we can catch the shadow. The task of catching the shadow on $X$ in this case is equivalent to catching the robber. But there are some game spaces with a rich group of isometries, and catching the shadow of some of these isometries may be easier than catching the robber.

Let us observe that the same shadow strategy applies in the $\varepsilon$-approaching games, where the goal is to approach the robber to within distance $\varepsilon$ from the optimal value $\ValC$. In this paper, we will only use it for the case when $\ValC=0$, thus we will talk about being at distance $\varepsilon$ from the robber but the setup applies also for the case when $\ValC>0$. For the $2\varepsilon$-approaching game on $X$, we can use $\varepsilon$-approaching strategy on $Y$ and with this we can \emph{$2\varepsilon$-guard} the set
$$
    \Phi_\varepsilon(\sigma) := \{ x\in X \mid d(x,\sigma(x))\le \varepsilon \}.
$$
Again, $2\varepsilon$-guarding the set means that whenever the robber steps on this set, there will be a cop at distance at most $2\varepsilon$ from him, thus achieving the goal in the $2\varepsilon$-approaching game.

Suppose that $\sigma$ is a 1-Lipschitz mapping that is of bounded order on a subset $Z\subseteq Y$, where $Y=\sigma(X)$. More precisely, let us assume that for each $z\in Z$ there is an integer $m$, $1\le m\le p$, such that $\sigma^m(z) = z$. Suppose that we have $p+q-1$ cops $C_1,\dots,C_{p+q-1}$ and that we are able to catch the robber in $Y$ with $q$ cops. Then we can do the following:

1) Catch the shadow of the robber and use one cop (say $C_1$) to stay in the shadow.

2) Next, we catch the shadow of $C_1$ using $q$ cops different from $C_1$, while $C_1$ follows the robber. Then we keep another cop, say $C_2$, in the shadow of $C_1$.

3) Continue this process with the remaining cops so that at the end, the cop $C_{i+1}$ is in the shadow of $C_i$ for $i=1,\dots,p-1$.

After we reach this situation, whenever the robber steps in $Z$, he will be caught because if his position is $z$, then we have cops at positions $\sigma^m(z)$ for $m=1,2\dots,p$, and one of these points is equal to $z$. Thus, we can guard $Z$.

\section{Geodesic surfaces and their boundary}

In the rest of the paper we will develop bounds on cop numbers for compact geodesic surfaces.
Topological classification of surfaces tells us that every compact surface is homeomorphic to a closed surface $S_{g,k}$ or $N_{g,k}$ of finite genus $g\ge0$ (either orientable or nonorientable, respectively) from which a finite collection of $k\ge0$ open disks (whose closures are pairwise disjoint disks) is removed.

When dealing with the game on geodesic surfaces, we will use Theorem \ref{thm:triangulate surface epsilon} to triangulate the surface and will therefore need to assume that the boundary is piecewise geodesic. In order to do that, we will first replace the surface $X$ with a different surface $X'$, whose boundary will be piecewise geodesic. Formally, this will be achieved by the following results.

\begin{theorem}\label{thm:obtaining geodesic boundary}
  Let $X$ be a compact geodesic surface and $\varepsilon>0$. Then $X$ contains a geodesic surface $X'\subseteq X$ that is homeomorphic to $X$ and has the following properties:
  \begin{itemize}
    \item[(i)] $X'$ is isometric in $X$.
    \item[(ii)] $X'$ has piecewise geodesic boundary.
    \item[(iii)] There is a $1$-Lipschitz mapping $\sigma: X\to X'$ that is identity on $X'$, and for each $x\in X\setminus X'$, we have $\sigma(x)\in \partial X'$ and $d(x,\sigma(x))<\varepsilon$.
  \end{itemize}
\end{theorem}

\begin{proof}
  The topological part of the proof can be found in the appendix, see Theorem \ref{thm:make boundary piecewise geodesic}. There it is shown that we may restrict our attention separately to each boundary component $B$ and it is shown that there is a finite set of points $w_1<w_2<\cdots <w_s < w_1$ and $z_1 < z_2 <\cdots <z_s < z_1$ on $B$ (where $<$ denotes the cyclic order around $B$) and there are $(w_j,z_j)$-geodesics $\gamma_j$ ($1\le j\le s$) with the following properties (for every $1\le j\le s$), where all indices are considered modulo $s$:

\begin{itemize}
    \item[(1)] The pairs $(w_j,z_j)$ and $(w_{j+1},z_{j+1})$ interlace on $B$, i.e.\ $w_j < w_{j+1} \le z_j < z_{j+1}$, and the union of all segments $B[w_j,z_j]$ covers $B$.
    \item[(2)] The $(w_j,z_j)$-geodesic $\gamma_j$ bounds a (possibly degenerate) disk $D_j$ together with $B[w_j,z_j]$, $\gamma_j$ is the unique geodesic from $w_j$ to $z_j$ contained in $D_j$, and $\gamma_j[w_j,z_j]\cap B \subseteq B[w_j,z_j]$.
    \item[(3)] The length of $\gamma_j$ is less than $r$, $\ell(\gamma_j)<r$, and the diameter of $D_j$ is smaller than $\varepsilon$. (Here the constant $r$ satisfies $0<r<\varepsilon/4$ is such that the diameter of each $D_j$ is less than $\varepsilon$.)
    \item[(4)] $\gamma_j$ intersects $\gamma_{j-1}$ and $\gamma_{j+1}$, but is disjoint from all other $\gamma_m$, $m\notin \{j-1,j,j+1\}$.
    \item[(5)] Let $x'_j$ be the first point on $\gamma_{j-1}$ that belongs to $\gamma_{j-1}\cap \gamma_j$, when $\gamma_{j-1}$ is traversed from $w_{j-1}$ towards $z_{j-1}$. Then the union $\cup_{j=1}^s \gamma_j[x'_j,x'_{j+1}]$ forms a simple closed curve $B'$ in $X$ that is homotopic to $B$.
\end{itemize}

The surface $X'$ is obtained from $X$ by removing (for each boundary component $B$) the (degenerate) cylinder between $B$ and $B'$ (while keeping points in $B'$). The above construction defines $X'$ and the proof of Theorem \ref{thm:make boundary piecewise geodesic} verifies properties (i) and (ii). It remains to prove (iii).

  \begin{figure}
  \centering
  \includegraphics[width=16cm]{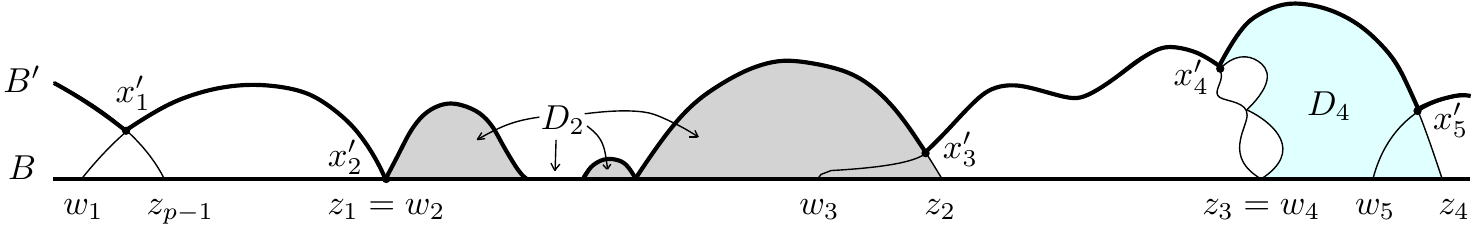}
  \caption{Some degenerate situations when defining $B'$.}\label{fig:defining B'}
  \end{figure}

Some degenerate situations of the above are shown in Figure \ref{fig:defining B'}: The disk $D_j$ can be ``degenerate'' as is the case with $D_2$ in the figure; the consecutive geodesics $\gamma_j$ and $\gamma_{j+1}$ can intersect more than once; however the latter case is restricted to look like the case of $\gamma_3$ and $\gamma_4$ in Figure \ref{fig:defining B'} because of property (2) which implies that $D_j$ cannot be made smaller. If that happens, we replace $\gamma_j$ with $\gamma_j[w_j,x'_{j+1}] \cup \overline{\gamma}_{j+1}[x'_{j+1},w_{j+1}]$, where $\overline{\gamma}_{j+1}$ denotes the path that is reverse to $\gamma_{j+1}$. In that case we also change $D_j$ correspondingly.

Since each $\gamma_j$ is a geodesic, the following defines a 1-Lipschitz mapping $\sigma_j:D_j\to X'$:
$$
   \sigma_j(x) = \left\{
                       \begin{array}{ll}
                         z_j, & \hbox{if $d(w_j,x)\ge d(w_j,z_j)$;} \\
                         \gamma_j(t), & \hbox{where $d(w_j,x)=d(w_j,\gamma_j(t))$, otherwise.}
                       \end{array}
                     \right.
$$
Note that for each $x\in X\setminus X'$, there is $j\in [s]$ such that $x$ belongs to $D_j$ and possibly also to $D_{j+1}$, but not to any other $D_i$, $i\in [s]\setminus\{j,j+1\}$. In that case we define
$$
   \sigma(x) = \left\{
                       \begin{array}{ll}
                         \sigma_j(x), & \hbox{if $\sigma_j(x)\in \gamma_j[x'_j,x'_{j+1}]$;} \\
                         x'_j, & \hbox{if $\sigma_j(x)\in \gamma_j[w_j,x'_j]$;} \\
                         x'_{j+1}, & \hbox{if $\sigma_j(x)\in \gamma_j[x'_{j+1},z_j]$.}
                       \end{array}
                     \right.
$$
This defines $\sigma(x)$ for every $x\in X\setminus X'$, and for $x\in X'$, we define $\sigma(x)=x$.
To see that the mapping $\sigma$ defined above is 1-Lipschitz, note first that $\sigma$ restricted to $D_j\setminus D_{j-1}$ is 1-Lipschitz. Also, $\sigma|_{X'} = id_{X'}$ is 1-Lipschitz since $X'$ is isometric in $X$. The only thing to confirm is that it is continuous in the intersection of $D_j\setminus D_{j-1}$ and $D_{j+1}\setminus (int D_j)$.
To see this, consider a point $y\in \gamma_j[x'_{j+1},z_j]$. Then $\sigma(y) = x_{j+1}$. If we are approaching to $y$ from $D_{j+1}\setminus D_j$, the values of $\sigma_{j+1}$ approach $y$ and $\sigma$-values all become equal to $x'_{j+1} = \sigma(y)$ once the points are close enough to $y$. This completes the proof.
\end{proof}

The corollary for the game of cops and robber on surfaces is that for the $\varepsilon$-approaching game on $X$, we can always reduce our attention to the game played on the isometric homeomorphic subsurface $X'$ that has piecewise geodesic boundary.

\begin{corollary}\label{cor:cop number of X'}
  Let $X$, $\varepsilon>0$, and $X'$ be as in Theorem \ref{thm:obtaining geodesic boundary}. Then
  $$c_{2\varepsilon}(X) \le c_\varepsilon(X') \le c_\varepsilon(X).$$
\end{corollary}

\begin{proof}
Let $\sigma: X\to X'$ be the 1-Lipschitz mapping from Theorem \ref{thm:obtaining geodesic boundary}.
If $x\in X$ is a position of a player on $X$, then we may consider player's \emph{shadow} $\sigma(x)\in X'$. Since $\sigma$ is 1-Lipschitz, the movements of the shadows can be viewed as the game played in $X'$.

Suppose that $k=c_\varepsilon(X')$ cops are playing against the robber on $X$. The cops can use their winning strategy for $\varepsilon$-approaching game on $X'$ against the shadow $\sigma(r)$ of the robber. When they are at distance at most $\varepsilon$ from $\sigma(r)$ in $X'$, their distance from $r$ in $X$ is at most $2\varepsilon$ because $d(r,\sigma(r))\le\varepsilon$. This shows that $c_{2\varepsilon}(X) \le c_\varepsilon(X')$.

To prove the other inequality, consider the robber using his strategy in $X'$ for staying more than $\varepsilon$ away from the shadows of $k-1 = c_\varepsilon(X')-1$ cops. Since $\sigma$ is 1-Lipschitz, any cop $C$ at any time of the gameplay must be more than $\varepsilon$ away from the robber since $d(C,r) \ge d(\sigma(C),\sigma(r)) = d(\sigma(C),r) > \varepsilon$. This implies that $c_\varepsilon(X) \ge c_\varepsilon(X')$.
\end{proof}

\section{Surfaces of genus 0 and 1}

Aigner and Fromme \cite{AiFr84} proved that the cop number of any planar graph is at most 3. Abrahamsen et al.~\cite{AHRW17} extended their result to arbitrary compact subsets of the plane with rectifiable boundary components. Their version of the game uses the intrinsic metric induced by distances in $\RR^2$ and constant agility functions (with the step length tending to 0). Here we extend their result to arbitrary geodesic spaces that are homeomorphic to a compact subset of the sphere and endowed with arbitrary geodesic intrinsic metric. They are topologically determined by the number $k\ge0$ of boundary components, but the intrinsic metrics making them into geodesic spaces can be very different and complicated. Locally we may have lots of ``hills and valleys'', the curvature can be positive or negative and need not be bounded, the surface itself may not be intrinsically isometrically embeddable in $\RR^3$.

\begin{theorem}\label{thm:genus 0}
  If $X$ is a compact geodesic surface that is topologically homeomorphic to a subset of the $2$-dimensional sphere,
  then $c(X)\le3$.
\end{theorem}

By Corollary \ref{cor:epsilon-approaching suffices}, it suffices to prove the following.

\begin{lemma}
    If $X$ is a compact geodesic surface that is topologically homeomorphic to a subset of the $2$-dimensional sphere,
    then for every $\varepsilon>0$, we have $c_\varepsilon(X)\le3$.
\end{lemma}

\begin{proof}
  By Corollary \ref{cor:cop number of X'}, we may assume that $X$ has piecewise geodesic boundary.
  Let us now consider the $\varepsilon$-approaching game with 3 cops.
  By Corollary \ref{cor:cop number of X'} we may assume that each boundary component of $X$ is piecewise geodesic.
  Since $X$ is compact, it has only finitely many boundary components. If the boundary is nonempty, let $\delta$ be the maximum distance from a point in $X$ from the boundary of $X$ and let $p_0\in X$ be a point whose distance from the boundary is at least $2\delta/3$. We may assume that $\varepsilon<\delta/3$.
  By Theorem \ref{thm:triangulate surface epsilon}, $X$ is the union of a finite collection $\mathcal T$ of non-overlapping triangles, and each triangle in $\mathcal T$ is convex and non-degenerate, its diameter is less than $\varepsilon/2$, and if a vertex $p$ of one of these triangles is a corner, then $p\in\partial X$. Let us denote by $\VT$ and $\ET$ the set of vertices and edges (respectively) of triangles in $\T$. (We will consider $\VT$ and $\ET$ as the set of vertices and edges but also as subsets of $X$. Thus, if we say that a point $p\in X$ is in $\ET$, we mean it is a point on one of the edges (geodesic curves) in $\ET$.) We may assume that $p_0\notin \VT\cup\ET$ is an interior point of one of the triangles, say $T_0\in\T$. Let $x_0,y_0,w_0$ be the vertices of $T_0$. The above assumptions guarantee that $T_0$ is disjoint from $\partial X$. Now we view $X-p_0$ as a subset of the plane and view $\T$ as a planar graph, whose outer face corresponds to $T_0$. This gives a homeomorphism of $X-p_0$ with the plane with $k$ open disks removed. Note that this is a topological representation of $X$ and that we still consider the geodesics in $X$ defining the distances between the points in $X$.

  Given a closed curve $\gamma$ in $X$ which is piecewise geodesic and is not crossing itself (but may touch itself), we define the interior $\int(\gamma)$ and the exterior $\ext(\gamma)$ of $\gamma$ in the same way as we define it for curves in the plane, except that we take the intersection with $X-p_0$. Thus the holes in the surface are not included in $\int(\gamma)$ ($\ext(\gamma)$), but their boundary or parts of their boundary may be included in $\int(\gamma)$ ($\ext(\gamma)$). And, to clarify, note that $\gamma\cap\int(\gamma)=\emptyset$ and $\gamma\cap\ext(\gamma)=\emptyset$.

  The strategy of three cops chasing the robber in $X$ will involve the following configuration using two geodesic paths in $X$.
  (In fact, a weaker condition -- see the last item in the definition of a cage below -- will be imposed on these two paths.)
  One cop will guard a path $P$ joining a point $x\in\VT\cup\ET$ with a point $y\in\VT\cup\ET$.
  The second cop will guard a path $Q$ from a point $w\in\VT\cup\ET$ to a point $z\in\VT\cup\ET$. A closed curve $F$ in $X$ containing $P$ and $Q$ will be called a \emph{cage for the robber} if the following conditions hold:

\begin{figure}
  \centering
  \includegraphics[width=12cm]{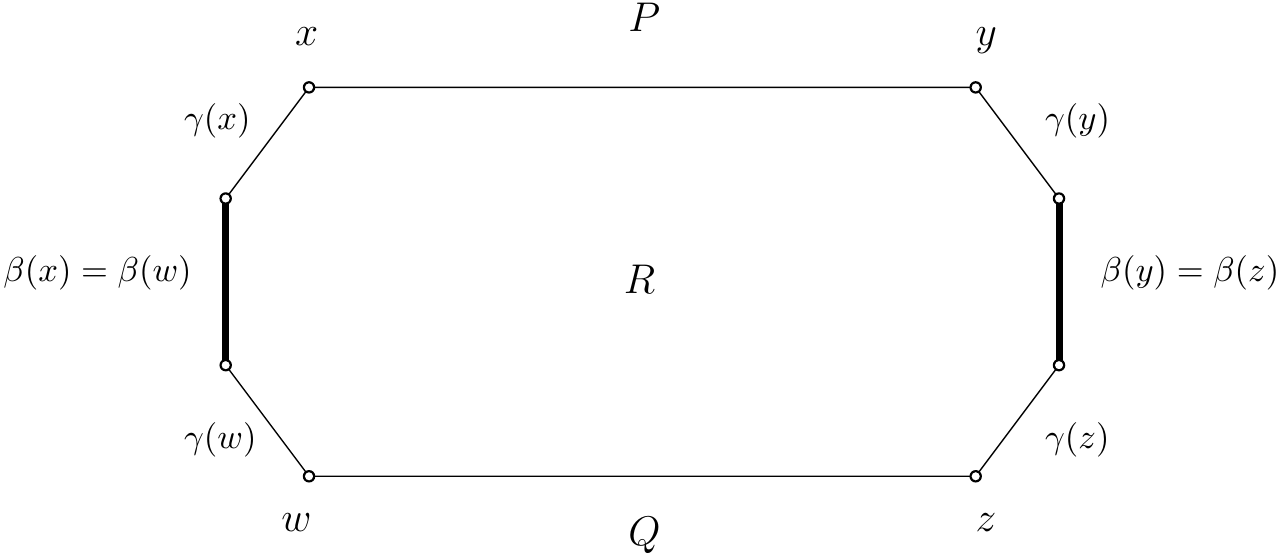}
  \caption{A cage for the robber encloses robber's territory $R$. Each cage is bounded by two ``internally geodesic'' paths $P$ and $Q$ (geodesic inside $\overline R$), four short geodesics $\gamma(u)$, with $\ell(\gamma(u)) \le \varepsilon/2$, $u\in\{x,y,w,z\}$, and two boundary segments $\beta(x)=\beta(w)$ and $\beta(y)=\beta(z)$ (shown as bold lines). Each of the constituents can be a single point. Moreover, $\beta(x)$ (and $\beta(y)$) can be empty, in which case $\gamma(x)=\gamma(w)$ ($\gamma(y)=\gamma(z)$, respectively). Each of the points depicted as a small circle is either a vertex in $\VT$ or a point on an edge in $\ET$. The paths $P$, $Q$, and $\gamma(u)$ need not be contained in the 1-skeleton of $\T$.}
  \label{fig:cage for the robber}
\end{figure}

  \begin{itemize}
    \item $F$ consists of $P\cup Q$ and of up to six additional paths as shown in Figure \ref{fig:cage for the robber}.
    \item Among the six additional paths contained in $F$ are two boundary segments $\beta(x)=\beta(w)$ and $\beta(y)=\beta(z)$, whose endpoints are joined to points $x,w$ and $y,z$ (respectively), as shown in Figure \ref{fig:cage for the robber}, by four \emph{short geodesics} $\gamma(u)$, $u\in\{x,y,w,z\}$, each of which is contained in a single triangle in $\T$ (and hence the length of each of them is at most $\varepsilon/2$). If $\gamma(u)$ contains more than just one point, then the points in $\int(F)$ close to $\gamma(u)$ are in $X$, i.e., $\gamma(u)$ is not part of a boundary of a hole inside $\int(F)$. Each boundary segment ($\beta(x)$ and $\beta(y)$) may be a single point on one of the boundary components and may also be absent (empty), in which case the corresponding two short geodesics coincide (i.e., $\gamma(x)=\gamma(w)$ or $\gamma(y)=\gamma(z)$).
    \item The points $x,y$ are endpoints of $P$ and $w,z$ are the endpoints of $Q$.
    \item The closed curve $F = P\cup\gamma(y)\cup\beta(y)\cup\gamma(z) \cup Q\cup\gamma(w)\cup \beta(w)\cup\gamma(x)$ does not cross itself (but may touch itself either inside or outside of the idealised drawing in Figure \ref{fig:cage for the robber}), and encloses a region $R=\int(F)$ containing the position of the robber. The region $R$ will be called the \emph{territory of the robber}. Let $\overline R = R\cup F$ be the closure of $R$.
    \item The paths $P$ and $Q$ are \emph{geodesic in $\overline R$}, by which we mean that no path in $\overline R$ joining two points of $P$ (or two points in $Q$) can be shorter than the subpath on $P$ (subpath on $Q$, respectively) joining the same two points.
  \end{itemize}

  Just for clarity, let us repeat that $R$ is an open set, $x,y,w,z$ are points in $\VT\cup\ET$, while $P$ and $Q$ need not be contained in the 1-skeleton of $\T$. See also the remarks at the caption of Figure~\ref{fig:cage for the robber}.

  Let us first show that if one cop guards $P$ and another path guards $Q$ (and guarding is defined with respect to the intrinsic distance in $\overline R$), then the robber cannot escape from the cage. More precisely, if the robber ever steps on $F$, one of the cops will be at distance at most $\varepsilon$ and so the cops will win the $\varepsilon$-approaching game, or he will be on $\beta(x)\cup\beta(y)$ and unable to exit $\overline R$. If the robber steps on $P\cup Q$, he will be caught since these two geodesics are guarded; if he steps on $\beta(x)\cup\beta(y)$ then he either has to step on some $\gamma(u)$ before, or he will not be able to leave $\overline R$ since $\beta(x)$ and $\beta(y)$ are both segments on the boundary of $S$. Thus it suffices to see that he cannot step on $\gamma(x)$ (the proof for $\gamma(u)$, $u\in\{y,w,z\}$, is the same). Since $\gamma(x)$ has length at most $\varepsilon/2$, any point on $\gamma(x)$ is at distance $t\le\varepsilon/2$ from $x$. The distance from $x$ of the cop who is guarding $P$ must be less or equal to $t$ in order to prevent the robber to reach $x$ without being caught. This means that the distance from that cop to the robber is at most $\varepsilon$.

  To start, we need to find the initial cage for the robber. To get it, we take the triangle $T_0=x_0y_0w_0$ in $\mathcal T$ and let $P$ be the edge $x_0y_0$ and $Q$ be the edge $w_0y_0$ (so that $x=x_0$, $y=z=y_0$, and $w=w_0$) of this triangle. The edge $x_0y_0$ serves as $\gamma(x)=\gamma(y)$ and the boundary segments $\beta(x)$ and $\beta(y)$ are empty.
  This configuration clearly forms a cage for the robber, where $P$, $Q$ and $\gamma(x)$ are the only nontrivial ingredients of the cage. When we bring two cops on $P$ and $Q$ (respectively), the robber must be in the region bounded by the boundary of the triangle $T_0$ since all points outside are in $T_0$ and hence at distance at most $\varepsilon$ from any point on $P\cup Q$. Next, we make sure that the two cops on the paths guard $P$ and $Q$ according to the shadow strategy. This is how we start.

  The outline of the proof is now as follows. Having and guarding a cage for the robber, we will show that we can change the cage and \emph{shrink robber's territory} if $R$ has nonempty intersection with at least one edge in $\ET$. By \emph{shrinking} we mean that the new territory of the robber intersects fewer edges in $\ET$. Once $R$ intersects no edges, $R$ will be contained in the interior of a single triangle of $\mathcal T$. Since that triangle has diameter less than $\varepsilon$, the robber is at distance at most $\varepsilon$ from the boundary of $R$, and since there is a cop on $P$, the cops win the $\varepsilon$-approaching game.

  In the first step we change the cage so that it is \emph{tame}, which means having the following properties:

  \begin{itemize}
    \item If $\beta(x)\ne\emptyset$, then the hole containing $\beta(x)$ lies in $\ext(F)$; the same holds for $\beta(y)$.
    \item If an edge $e$ in $\ET$ intersects $F$ and also contains a point in $R$, then one of the vertices of this edge is in $R$.
    \item We neglect the possibility that $F$ may touch itself in $\ext(F)$. If it does, we cut the surface along the touching points or touching segments and consider the touching parts to be disjoint after cutting. With this convention, $F$ is a simple closed curve.
  \end{itemize}

  Having a cage for the robber, we first tame it. The \emph{taming process} is as follows. Suppose that $\beta(x)\ne\emptyset$ and let $H$ be the hole in the interior of $F$ whose boundary $B$ contains $\beta(x)$. Let $\beta'(x)=B\setminus \int(\beta(x))$. By replacing $\beta(x)$ with $\beta'(x)$, $F$ changes into another cage $F'$, and the hole $H$ moves to the exterior of $F'$. Note that $\beta'(x)$ may contain the whole boundary component (if $\beta(y)$ was a single point), but this is not forbidden. The new cage shrinks robber's territory. Note that the condition on $\gamma(x)$ and $\gamma(w)$ from the second property of cages guarantee that $\beta'(x)$ does not intersect internal points of $\gamma(x)$ or $\gamma(y)$. This establishes the first property of tame cages.

  \begin{figure}
  \centering
  \includegraphics[width=12cm]{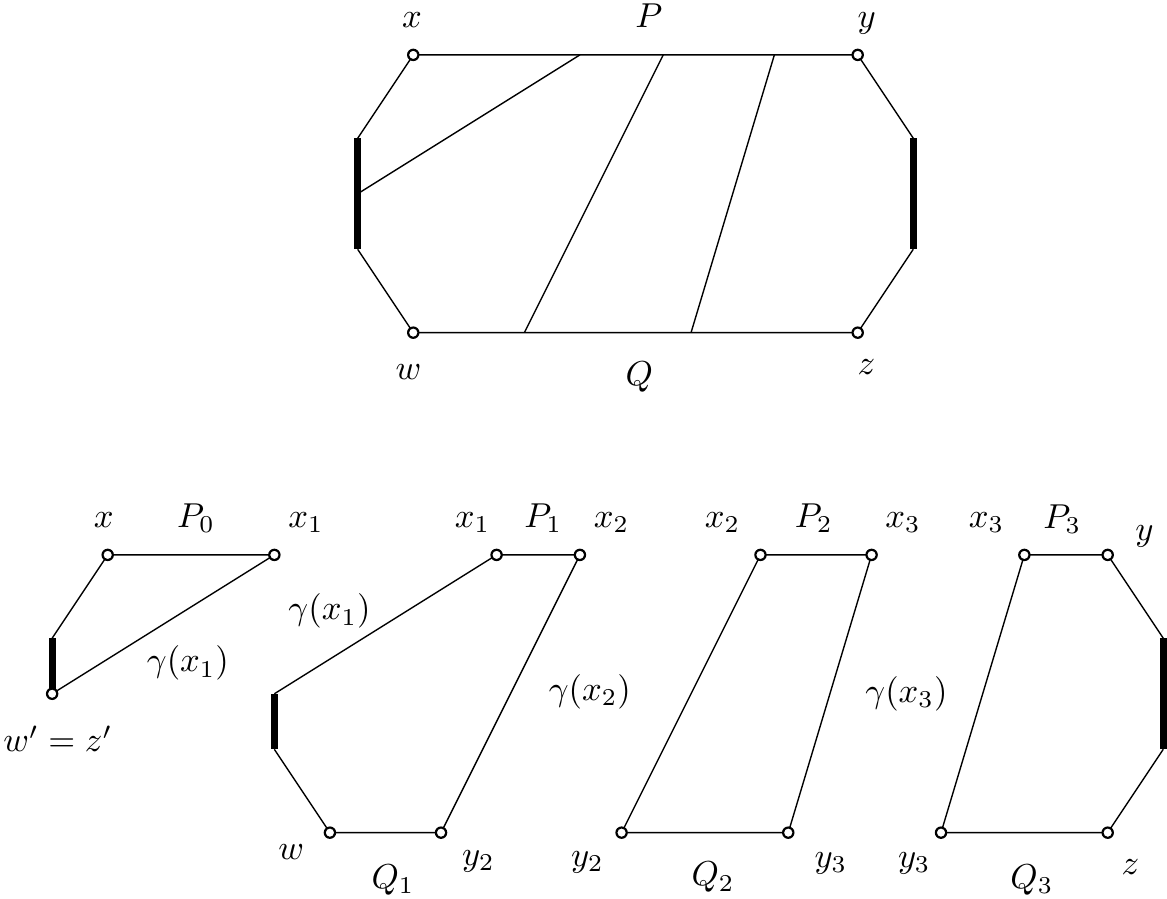}
  \caption{Chords from $P$ to $Q\cup \beta(x)\cup \beta(y)$ split the cage into subcages.}\label{fig:chords split cage}
  \end{figure}

  To establish the second tame condition, consider any edge having a segment $S$ joining two points $a,b$ in $F$ and otherwise being contained in $R$. Then $S$ splits $R$ into two regoins and is called a \emph{chord} in $F$. If we were able to replace our cage with another cage containing the region with the robber and that region would not contain the edge $e$ anymore, we would shrink robber's territory. The technical problem here is that the same edge can have many (possibly infinitely many) chords in $F$. Let us consider the components of the intersection of $e$ with $P$. If we traverse $e$, these components proceed on $P$ monotonically in the direction from $x$ towards $y$ (or vice versa), since both $P$ and $e$ are geodesics. By replacing the segment of $P$ between any two consecutive components of $P\cap e$ by the corresponding segment of $e$, we obtain another $(x,y)$-geodesic $P'$ containing all chords that connect two points on $P$. We can bring the third cop and guard this geodesic. If the robber is in one of the components bounded by $P$ and $P'$, then we can replace the cage with the cage consisting of parts of $P$ and $P'$ and then release the cop who was previously guarding $Q$. This clearly shrinks robber's territory. On the other hand, if the robber is not in any such component, we may replace $P$ with $P'$. Now, we may assume that $e$ no longer has a chord joining two points on $P$. Similarly, no chord joins two points on $Q$. Note that $e$ cannot intersect the internal points on any $\gamma(u)$, $u\in \{x,y,w,z\}$ by the definition of the cage. But it can have a point on $\beta(x)$. Consider now all chords of $e$ in $R$ that contain a point in $P$. The other endpoints of such chords are in $Q\cup \beta(x)\cup \beta(y)$. Any such chord from a point $a\in P$ to a point $b\in Q\cup \beta(x)\cup \beta(y)$ splits the cage into two subcages (the left and the right side as shown in Figure \ref{fig:chords split cage}, where we have three chords and the first, the last and any two consecutive chords give rise to subcages). We may have infinitely many chords but any two consecutive chords form a smaller cage and we only need the one that contains the robber. Since two cops guard $P$ and $Q$, they also guard all these smaller subcages in which one or two chords of $e$ play the role of short segments $\gamma(u)$. Note that each chord has length less than $\varepsilon/2$, thus the argument given above that the robber cannot escape from a cage shows that the robber will have to stay in the cage he is currently in.

  Lastly, we argue how to make $F$ be simple. By our convention, $F$ does not touch itself in the exterior, but it may touch itself one or multiple times (even infinitely many times) in $\overline R$. This case works in the same way as the case with the chords treated above since the self-touchings split $R$ into subcages.

  \begin{figure}
  \centering
  \includegraphics[width=14cm]{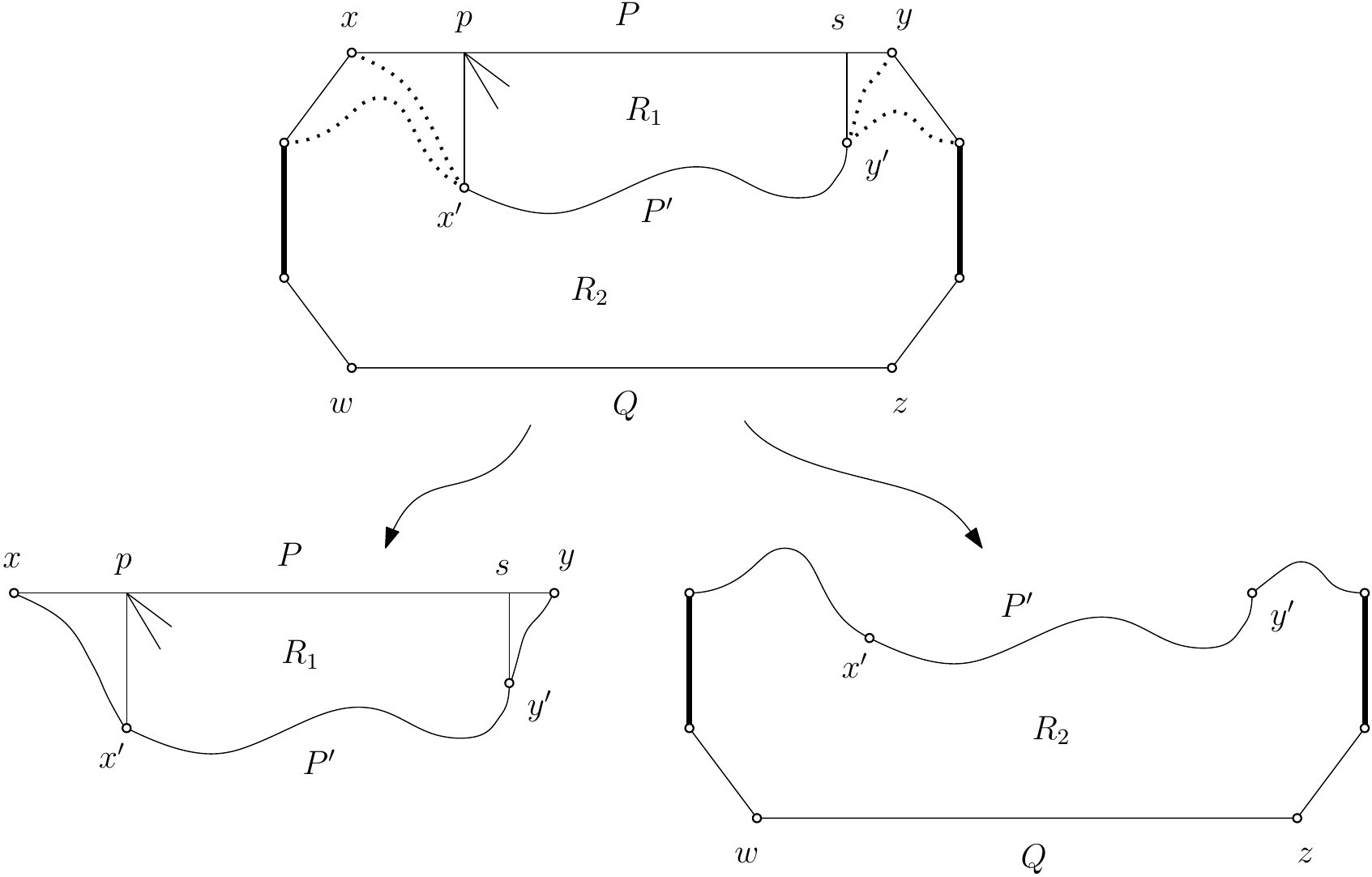}
  \caption{To win the $\varepsilon$-approaching game on a surface of genus 0, we guard the territory of the robber with two cops forming a cage and make the territory smaller by using the third cop.}\label{fig:planar strategy}
  \end{figure}

  After the cage is tame, the proof of the general step is simple. The step to shrink robber's territory is as follows. Starting at $x$, we find the first point $p\in P$ such that $p \in \VT\cup\ET$ and from this point we have an edge leading into $R$. If there are more than one such edges, we take the one that is closest in the clockwise order around $p$ to the face containing $\gamma(x)$ (with the obvious meaning when $\gamma(x)$ is just a point). Then we traverse this edge from $p$ into $R$ until we reach a vertex $x'$. This vertex is in $R$ since the cage is tame. The traversed edge is in the same triangle $T\in \T$ as $\gamma(x)$, hence the distance of $x'$ from the point $t=\gamma(x)\cap\beta(x)$ (or $t=w$ when $\beta(x)=\emptyset$) is at most $\varepsilon/2$. We let $\gamma(x')$ be a geodesic from $x'$ to $t$. Since the triangle $T$ is convex, $\gamma(x')$ is contained in $T$. We repeat the similar process at $y$, obtaining vertex $y'$ and a short geodesic from $y'$ to $\gamma(y)\cap \beta(y)$. See Figure \ref{fig:planar strategy}, where the geodesics are shown with dotted lines. We also consider geodesics from $x'$ to $x$ and from $y'$ to $y$ as shown in the figure. The cage is now split into two subcages. There is part of the triangle $T$ on the left and part of the triangle on the right between the two dotted short geodesics that is not covered by these subcages. But each of these parts is contained in a single triangle and is thus at distance at most $\varepsilon/2$ from $x$ (or from $y$); thus, if the robber is there, the cop guarding $P$ is at distance at most $\varepsilon$ from the robber. Now we bring the third cop to guard the $(x',y')$-geodesic $P'$ in $\overline R$. Once this is established, the robber is confined in one of the two subcages and we can now release one of the cops and continue shrinking robber's territory.

  One detail has to be added. Namely, the point $p$ may not exist. In that case, either $P\cup\gamma(x)\cup\gamma(y)$ is contained in single triangle $T_1\in \T$, or there is a ``hole" inside the disk bounded by $F$, and $P$ is on the boundary of that hole. In the latter case, $\beta(x)=\emptyset$ since otherwise $\beta(x)$ would be part of the boundary of the same hole inside $F$ (because of the assumption that $\varepsilon < \delta/3$) and this would contradict tameness of the cage. Thus, we may assume that $T_1$ exists. When this happens, we try to use the geodesic $Q$ instead of $P$. Again, if we are unsuccessful in shrinking the territory of the robber, then $Q\cup\gamma(w)\cup\gamma(z)$ is contained in single triangle $T_2\in \T$. So, we have both triangles $T_1$ and $T_2$. If $R\subseteq T_1\cup T_2$, the $\varepsilon$-approaching game stops with cops winning. Thus, we may assume that there are edges and vertices of $\T$ inside $R$. Since the graph $\VT\cup\ET$ is connected, there is a point $p\in \beta(x)$ (say) that is incident with an edge $e\in\ET$ leading into $R$. Let $x'$ be the endvertex of $e$ in $R$. As before, we choose $p$ and $e$ as close as possible to $x$ so that $x$ and $x'$ are in the same triangle in $\T$. We do the same from the other side (starting at $y$ and clockwise around $F$ until we get a point $s$ with an edge leading into $R$. We now proceed in the same way as in the generic version
  (Figure \ref{fig:planar strategy}) to shrink robber's territory.

  We have shown that we either end the game or succeed in shrinking robber's territory. This completes the proof.
\end{proof}

It is interesting that the strategy of cages used in the above proof can also be used for surfaces of genus 1. For the proof we need to recall some definitions.

Given a geodesic surface $X$, its \emph{systole}, denoted by $\sys(X)$, is defined as the least length of a noncontractible loop in $X$. Similarly, the \emph{essential systole}, $\sys_0(X)$, is defined as least length of a loop in $X$ that is noncontractible in the capped surface $\widehat X$ (when we cap off all boundary components of $X$).

\begin{theorem}\label{thm:genus 1}
  If $X$ is a compact geodesic surface that is topologically homeomorphic to a subset of the 2-dimensional torus,
  then $c(X)\le3$.
\end{theorem}

\begin{proof}
  Again, we may assume that each boundary component of $X$ is piecewise geodesic.
  Let $\widehat X$ be the torus obtained by ``capping off'' all boundary components and using a metric on these caps so that the intrinsic metric in $\widehat X$ induces the original intrinsic metric in $X$.
  To start with, we first cut $\widehat X$ along a shortest noncontractible closed curve $\gamma_1$, which is easily seen to be a closed geodesic. After cutting, we obtain a cylinder (with additional holes of $X$ inside), whose boundary consists of the two copies of $\gamma_1$. Next, we take a shortest geodesic $\gamma_2$ joining the two copies of $\gamma_1$ in the cylinder. By cutting along $\gamma_2$, we obtain a fundamental polygon of $\widehat X$. By tessellating the plane with copies of this fundamental polygon, we obtain the universal cover of $\widehat X$. Our assumption on the metric in the capped discs implies that $\gamma_1$ and $\gamma_2$ are both contained in $X$. Now we consider the cover $\widetilde X$ of $X$ obtained from the tessellation by removing all capped discs. Let $\pi:\widetilde X\to X$ denote the covering projection.

  We will play the game in $X$ but will make the strategy based on what happens in $\widetilde X$. For the initial positions of players in $X$ we lift them to the same fundamental polygon in $\widetilde X$. Now the cops look at what is going on in $X$. When any player in $X$ moves, his move is lifted to $\widetilde X$. In addition to this, the cops can at any time move their position $x$ in $\widetilde X$ into any other point $x'\in \pi^{-1}(\pi(x))$, but they will consider the robber's lifted position in $\widetilde X$ following the normal rules of the game. If the cops with this added ``teleporting'' ability can approach the robber in $\widetilde{X}$ to a distance $\varepsilon$, they will win the $\varepsilon$-approaching game in $X$. We fix a small $\varepsilon>0$ and consider $X$ dissected into a set of triangles $\T$ satisfying conditions of Theorem \ref{thm:triangulate surface epsilon}. We consider $\widetilde X$ to be triangulated with the set of triangles $\widetilde\T = \pi^{-1}(\T)$.

  Let $D=\diam(X)$ and $D_0=\sys_0(X)$, and let $n$ be an integer such that
  $$n \ge 1 + \frac{D}{D_0} \log_2(20nD\varepsilon^{-1}) .$$

  \begin{figure}
  \centering
  \includegraphics[width=12cm]{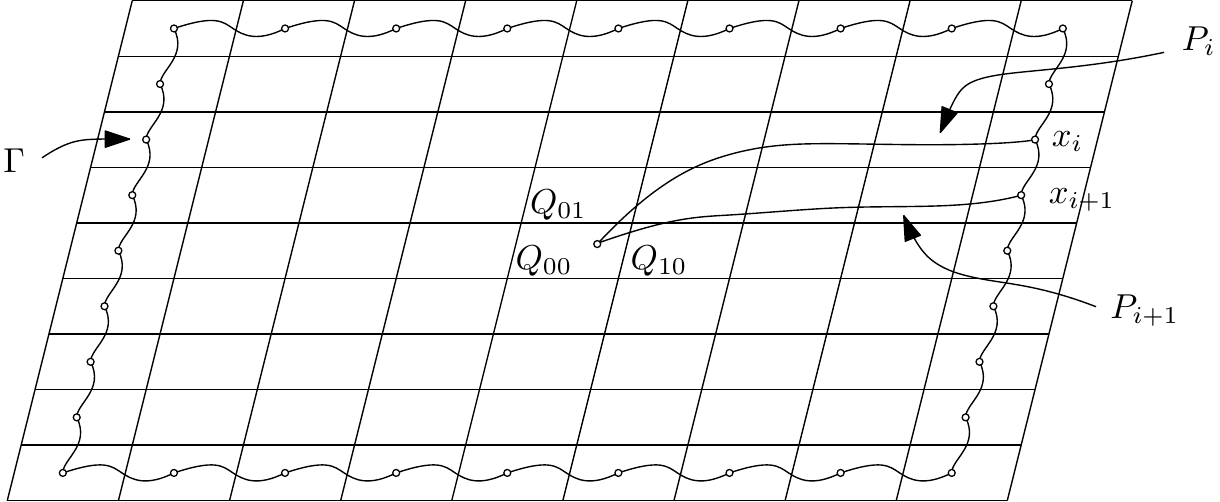}
  \caption{Lehner's trick to get the robber into a cage.}\label{fig:Lehner's trick}
  \end{figure}

  In the strategy of three cops in $\widetilde{X}$, we will use the trick of Lehner \cite{Le21} to obtain a bounded cage in $\widetilde X$ containing the position of the robber. The copies $Q_{ij}$ of the fundamental polygon can be enumerated by using integer pairs $(i,j)\in \ZZ^2$ by considering the way they tessellate the plane. Let $n$ be sufficiently large and consider a shortest closed curve $\Gamma$ through all polygons $Q_{ij}$ with $\max\{|i|,|j|\}=n$ that contains $Q_{00}$ in its interior. Let $N = \lceil 2\ell(\Gamma)/\varepsilon\rceil$ and let $x_1,x_2,\dots,x_N$ be points that appear on $\Gamma$ in this cyclic order such that $d(x_i,x_{i+1})\le \varepsilon/2$ for $i\in[N]$ (where $x_{N+1}=x_1$). Since
  $\ell(\Gamma) \le 4(2n+1)\max\{\ell(\gamma_1),\ell(\gamma_2)\}$ and $\ell(\gamma_j)\le 2D$ for $j=1,2$, we have
  $$ N = \lceil 2\ell(\Gamma)/\varepsilon\rceil \le 20nD\varepsilon^{-1}. $$
  Let $x_0\in Q_{00}$, and let $P_i$ be a shortest path from $x_i$ to $x_0$ in $\widetilde X$, $i\in[N]$. We may assume that these paths do not cross each other (but they may touch). The situation is sketched in Figure~\ref{fig:Lehner's trick}.

  Suppose that the robber is at distance $t$ from $x_0$. Let $P_i$ be any of the paths in our collection. Then we can teleport any cop to a point that is at distance at most $D$ from the point $p$ on $P_i$ whose distance from $x_0$ is equal to $t+D$. This cop can reach $p$ in time $D$ and at that time his distance from $x_0$ will be at least the distance of the robber from $x_0$. Thus the cop can start approaching $x_0$ along $P_i$ until he reaches the shadow of the robber on this path (see the proof of Lemma \ref{lem:isometric path}). Once he achieves this, he and the robber will be at distance at most $t+D$ from $x_0$. If the point $x_0$ is chosen to be the initial point of the robber in $Q_{00}$, then after catching the shadow of the robber $m$ times (for various paths $P_i$), the distance of the robber from $x_0$ will be at most $mD$. Now, three cops perform the bisection process. Two of them first catch the shadow on $P_1$ and $P_{n/2}$. Then, depending in which part of the disk $\int(\Gamma)$ is the robber, the third cop sets to guard the path in the middle. Once he guards that path, one of the cops can be released and the bisection process may continue. In $\log_2 N$ steps, there will be two cops guarding two consecutive paths $P_i$ and $P_{i+1}$. Since $d(x_i,x_{i+1})\le \varepsilon/2$, these two paths together with the short segment on $\Gamma$ form a cage for the robber, and now the planar strategy can be used to win the $\varepsilon$-approaching game in that cage.

  All that needs to be added is to verify that the robber cannot escape $\int(\Gamma)$ before two of the cops control consecutive paths. After $m = \lceil\log_2 N\rceil$ steps of the bisection, the robber is at distance at most $mD$ from $x_0$. In order to come from $Q_{ab}$ into adjacent polygon further away from $Q_{00}$, the robber needs at least $D_0$ time units. So, after $m$ steps, he will be in some $Q_{ab}$ with $\max\{|a|,|b|\} \le mD/D_0+1 < n$ and this polygon is still inside $\Gamma$.
\end{proof}

To summarize, we have the following result.

\begin{corollary}\label{cor:genus0and1}
  Let $X$ be a compact geodesic surface homeomorphic to a subset of the sphere, the torus, the projective plane, or the Klein bottle. Then $c(X)\le3$.
\end{corollary}

\begin{proof}
  Theorems \ref{thm:genus 0} and \ref{thm:genus 1} give the corollary for the two orientable surfaces. For the projective plane and the Klein bottle, we consider their double cover which is the sphere or the torus and let the three cops use the strategy on the covering space (lifting the initial position and the moves of the robber in $X$).
\end{proof}

\section{Cops and robbers on surfaces of higher genus}

Now we shall deal with arbitrary geodesic surfaces. First, we will provide lower bounds on the cop number. The main outcome in Section \ref{sect:lower bounds} is that there are surfaces of genus $g$ whose cop number is at least $\Theta(g^{1/2-o(1)})$, where the asymptotics for $o(1)$ is with respect to $g$. We will end up with an $O(g)$ upper bound.

\subsection{Lower bounds}
\label{sect:lower bounds}

In providing examples of surfaces whose cop number is large, we will use examples of metric graphs with large cop number combined with some special polyhedral discs that we will introduce next.

Let $d_1,\dots,d_k$ and $h$ be positive real numbers. For $i\in[k]$, let $Q_i$ be the rectangle with side lengths $d_i$ and $h$. Then we define the \emph{polygonal cylinder} $C=C(h;d_1,\dots,d_k)$ by consecutively identifying the sides of length $h$ of $Q_1,\dots,Q_k$, so that the ``left" side of $Q_i$ is identified with the ``right'' side of $Q_{i+1}$ for $i=1,\dots,k$, where $Q_{k+1} = Q_1$. The union $B$ of all ``bottom" sides of the rectangles used to obtain $C$ is called the \emph{base} of $C$. A similar construction, using a product $B\times [0,h]$, gives us a \emph{cylinder of height $h$ over $B$} for any rectifiable simple closed curve $B$. Equivalently, we can split $B$ into $k$ geodesic segments of lengths $d_1,\dots,d_k$ and use the same construction as described above.
Then the cylinder $C=B\times [0,h]$ is also a geodesic space that is isometric to the $\ell(B)\times h$ rectangle having its sides of length $h$ identified. Herewith, we will use the intrinsic $\ell_2$-metric on $C$, when viewed as the product of $B$ and the interval $[0,h]$ of length $h$:
$$
    d_C((b,t),(b',t')) = \left( d_B(b,b')^2 + (t-t')^2 \right)^{1/2}.
$$
For every point $x\in C$, we define the point $\pi(x)\in B$ as the closest point to $x$ in $B$ (the \emph{projection} to $B$), i.e. $\pi(b,t) = (b,0)$.

If we contract the set $B\times \{h\}$ into a point, we obtain the \emph{cone of height $h$ over $B$}. The cone $D = B\times [0,h]/(B\times\{h\})$ is a topological disk and we consider the local Euclidean distance in $D$, in which the length of $B\times\{t\}$ is equal to $(1-\tfrac{t}{h})\ell(B)$. If $h = \ell(B)/(2\pi)$, the cone is isometric to the disk of radius $h$ in the Euclidean plane and we call it the \emph{flat cone over $B$}.


\begin{figure}
  \centering
  \includegraphics[width=14cm]{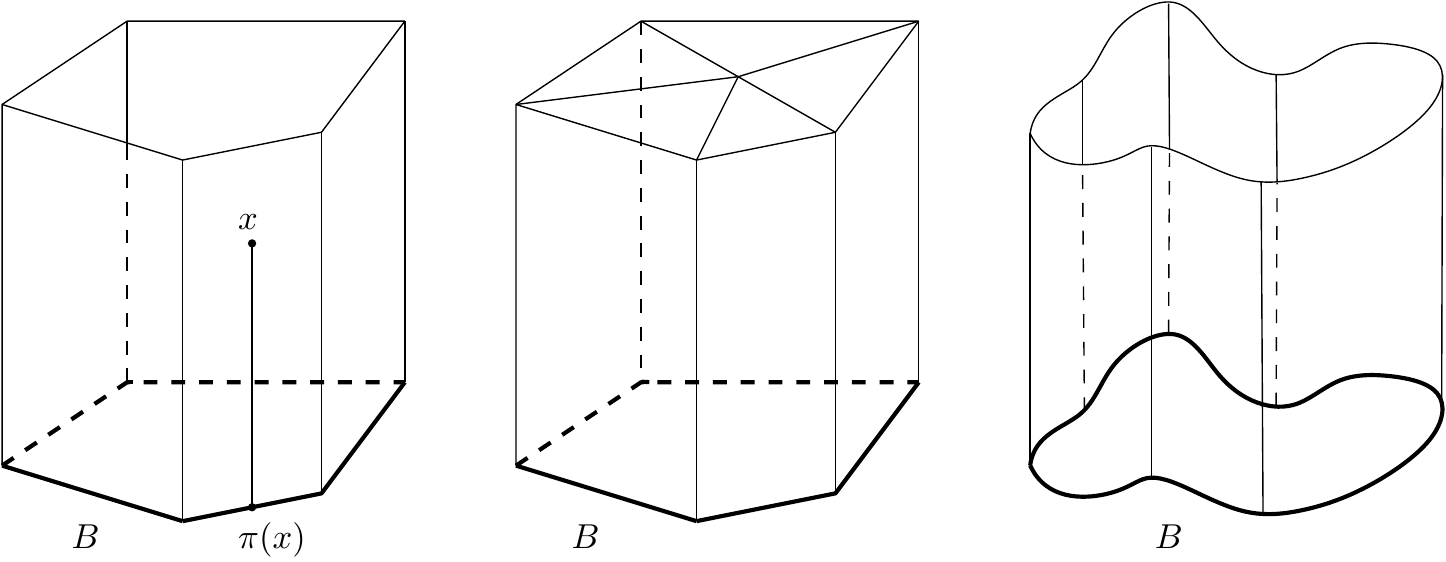}
  \caption{A polyhedral cylinder $C$ with base $B$ and height $h$, the capped cylinder $\widehat C$, and a cylinder over a closed curve. The apex point of the capped cylinder has the same distance $h'=\ell(B)/(2\pi)$ to each point at the top of the cylinder. The projection $\pi(x)$ of a point $x$ into $B$ is depicted.}\label{fig:polyhedral disk}
\end{figure}

If $X$ is a geodesic space and $B$ is a piecewise geodesic closed curve in $X$, we \emph{add the cylinder over $B$} by taking the union $X\cup C$, identifying $B\subset X$ with $B\times\{0\}\subset C$, and then combining the metrics into the intrinsic metric in the resulting geodesic space $\widehat X$. If we now \emph{cap off} the cylinder $C$ by adding a flat cone $D$ over $B'=B\times\{h\}$, and identify $B\times\{h\}$ with the base of the cone, we say that we have \emph{added a capped cylinder over $B$}, or simply that we have \emph{capped off} $B$. Usually we will use this term when $B$ is a boundary component of a geodesic surface $X$. See Figure \ref{fig:capped cylinder}.

\begin{figure}
  \centering
  \includegraphics[width=5cm]{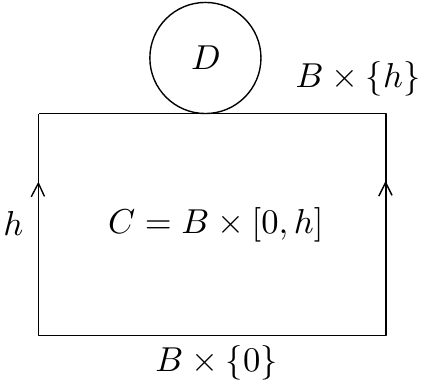}
  \caption{The intrinsic metric in a capped cylinder over $B$ uses the Euclidean metric of the $(\ell(B) \times h)$-rectangle in $C$ and the Euclidean distance within the disk $D$ of radius $\ell(B)/(2\pi)$. The left and right sides of $C$ are identified, and the top side of $C$ is identified with the perimeter of $D$.}\label{fig:capped cylinder}
\end{figure}

\begin{lemma}\label{lem:cylinder catch}
  Let\/ $Y= C\cup D$ be a capped cylinder over a piecewise geodesic boundary $B\subseteq X$, where $C$ is a cylinder of height $h$ and $D$ is the flat cone used to define the capped cylinder. Suppose that the height $h$ of the cylinder is at least $\ell(B)$ and that we have three cops positioned in $B$. Suppose that the robber is in $Y$ and if he is at a point $y\in C$, then one of the cops is at $\pi(y)$. Then the cops can catch the robber, and during the gameplay, the robber will not be able to come to $B$ without being caught.
\end{lemma}

\begin{proof}
  Note that the metric in $Y$ is locally Euclidean (see Figure \ref{fig:capped cylinder}), so the cages in the pursue of the robber will consist of two geodesic segments $P\cup Q$, without the need for added short geodesics. We may assume that $\ell(B)=2\pi$, so that $D$ is isometric to the unit disk and $C$ is isometric to the flat cylinder obtained from the rectangle $[0,2\pi]\times [0,h]$ by identifying the left and the right side.

  Having initial conditions stated in the lemma, two of the cops can position themselves on $B$ and guard $B$. In this way they make a cage, preventing the robber from entering $B$. Now we start with the strategy of shrinking the territory of the robber. First we keep 2 cops in $B$. If the robber is in $C$ or if he ever enters $C$ from the interior of $D$, these two cops can catch his shadow in $B$ before he could enter $B$. The third cop now positions himself so that he guards a geodesic from the center $z$ of the cap $D$ to a point $b_0\in B$. He will be at the same distance from $z$ as the robber all the time from now on. Reaching this situation, a single cop on $B$ will be able to guard $B$ from now on. Now the released cop can position himself so that he guards a geodesic from $z$ to the point $b_1\in B$ that is opposite to $b_0$ on $B$. Now the robber's territory is half smaller than before. From this point on, the three cops start using a strategy that has been discovered in \cite{IrMo22}, and we refer to that paper for the proof on how they eventually catch the robber.
\end{proof}

\begin{lemma}\label{lem:cylinder1}
  Let $X$ be a compact geodesic space with diameter $d$ and let $B\subset X$ be a piecewise geodesic simple closed curve of length $b$ in $X$. Let $\widehat X$ be the geodesic space obtained by adding a capped cylinder $C\cup D$ over $B$, whose height is greater or equal to $b+d$. If $c(X)\ge 3$, then $c(\widehat X) \le c(X)$ and $c_0(\widehat X) \le c_0(X)$.
\end{lemma}

\begin{proof}
The cops will play the game on $X$ trying to approach (or catch) the robber. If the robber moves into the cylinder to a point $y\in C$, they consider $\pi(y)$ as the position of the robber when applying their strategy. If the robber never enters the disk $D$ used to cap off the cylinder, the strategy will enable the cops to approach the shadow of the robber to within distance $\varepsilon$ (or catch the shadow when we consider $c_0(\widehat{X})$). When this happens, if the robber is at the same point as his shadow, the game is over and the cops win (either the $\varepsilon$-approaching game or catching the robber).

So, we may assume that either the robber is at a point $y\in C$ and there is a cop in $B$ at distance $\varepsilon$ from $\pi(y)$ (with $\varepsilon=0$ when catching), or that the robber is in $D$. In the latter case, the robber is so far from $B$ that we can bring another cop onto $B$ and achieve that two cops guard $B$ (since $B$ is composed of two isometric paths in $C$) before the robber is able to come close to $B$.

We may thus assume that $y\in C$ and there is a cop in $B$ at distance $\varepsilon$ from $\pi(y)$. From now on, the cop can keep this distance (unless the robber enters $D$ in which case $\pi(y)$ becomes undefined). In the continuation of the game we can bring another cop to $B$ and together with the first one they can catch $\pi(y)$. Moreover, they can position themselves on $B$ so that they guard $B$ from the robber entering it.

In any case, now the robber is in the topological disk $C\cup D$ and two cops guard its boundary $B$. Next, we bring the third cop into the cylinder and catch the robber by using the strategy of Lemma \ref{lem:cylinder catch}.
This shows that $c(\widehat X) \le c(X)$ and $c_0(\widehat X) \le c_0(X)$.
\end{proof}

In the proof of Lemma \ref{lem:cylinder1} we could argue that $c(X)$ and $c(\widehat X)$ are equal if $h$ is large enough, but this will be easier to do with a different capping arrangement, which will be described next.

Let us define the \emph{expanding cylinder of height $h$ over $B$} as the cylinder $C$ shown in Figure \ref{fig:expanding cylinder}, whose boundary consists of two concentric circles of radii $r_0=\ell(B)/(2\pi)$ and $r_1=r_0+h$. By \emph{capping this cylinder} we mean adding the Euclidean disk $D$ of radius $r_1$ and identifying the perimeters of $C$ and $D$.

\begin{figure}
  \centering
  \includegraphics[width=11cm]{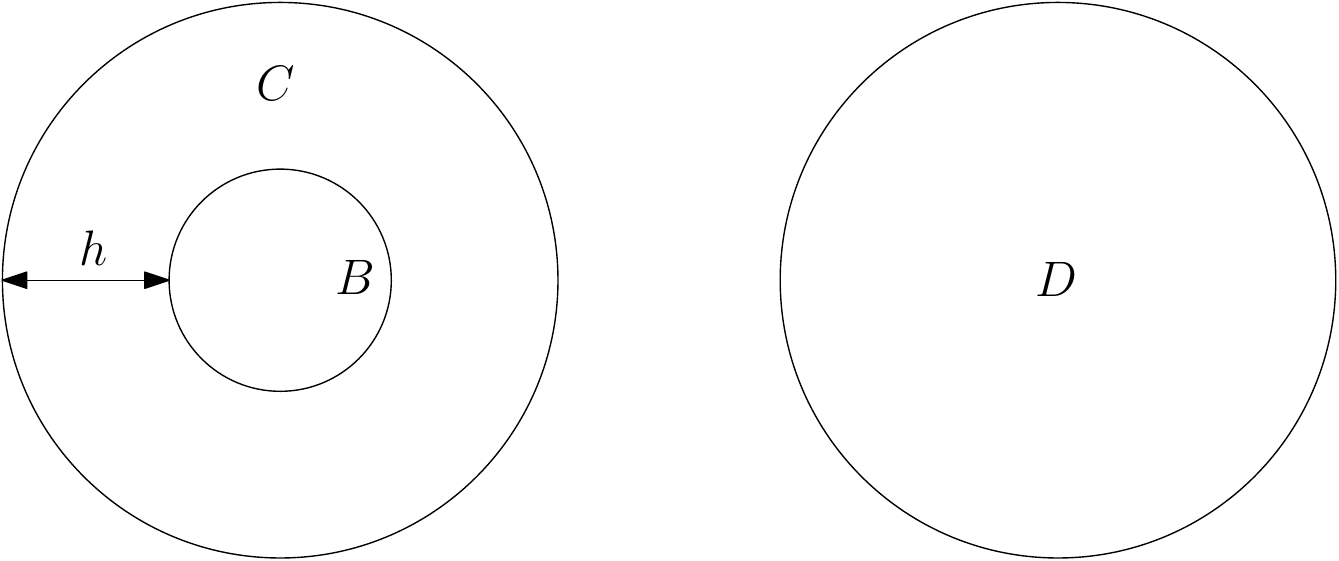}
  \caption{The intrinsic metric in the expanding cylinder of height $h$ over $B$ uses the Euclidean cylinder metric of the cylinder $C$ between the circles of radii $r_0=\ell(B)/(2\pi)$ and $r_1=r_0+h$, and uses the Euclidean distance within a disjoint disk $D$ of radius $(\ell(B)+h)/(2\pi)$ in the cap $D$.}\label{fig:expanding cylinder}
\end{figure}

\begin{lemma}\label{lem:cylinder1expanding}
  Let $X$ be a compact geodesic space with diameter $d$ and let $B\subset X$ be a piecewise geodesic simple closed curve of length $b$ in $X$. Let $\widehat X$ be the geodesic space obtained by adding a capped expanding cylinder $C$ over $B$, whose height is greater or equal to $b+d$. If $c(X)\ge 3$, then $c(\widehat X) = c(X)$ and $c_0(\widehat X) = c_0(X)$.
\end{lemma}

\begin{proof}
We prove that $c(\widehat X) \le c(X)$ and $c_0(\widehat X) \le c_0(X)$ in the same way as in Lemma \ref{lem:cylinder1}.

\begin{figure}
  \centering
  \includegraphics[width=7.5cm]{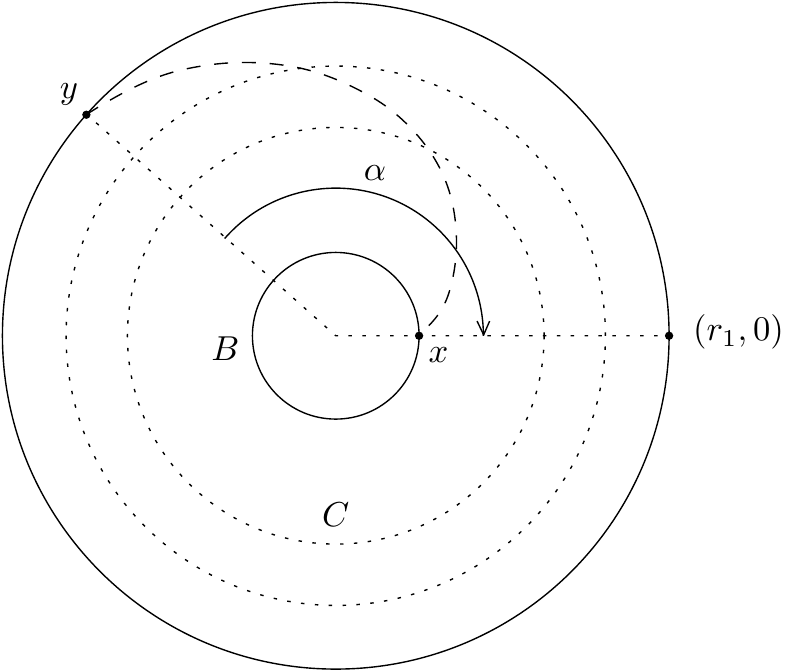}
  \caption{The schematics showing the 1-Lipschitz mapping $\psi:C\to B$ such that $\psi(y)=x$. The points that lie on the depicted dashed curve leading from $y$ to $x$ are all mapped to $x$. The shown dotted circles are first rotated by angle $\tfrac{3}{4}\alpha$ and $\tfrac{1}{2}\alpha$, respectively, and then radially projected onto $B$.}\label{fig:rotate by alpha}
\end{figure}

The converse inequality will be established by considering the robber, playing the game in $\widehat X$, mimicking his strategy in $X$. By this strategy, the robber will never enter the capped disk. There is an obvious 1-Lipschitz mapping $C\to B$ which will define the shadow of any cop in $C$. The robber will in his strategy stay away from the cops in $X$ and away from their shadows. The only nontrivial detail is what to do if a cop enters $D$. In that case, the goal of the cop must be to eventually come back into $X$ (or at least $\varepsilon$-close to $X$) in order to help catching the robber. When the cop enters $D$, his shadow $x\in B$ will be ``frozen'' until he returns to $C$ again; let $y\in C\cap D$ be the point of his entry into $C$. In order to use the shadow-escape strategy of the robber, it now suffices to see that for every point $y\in C\cap D$ and every point $x\in B$, there is a 1-Lipschitz map $\psi:C\to B$, which is identity on $B$ and satisfies $\psi(y)=x$. To see this, we may assume that the center of $C$ in the plane is the point $(0,0)$ and that $x=(r_0,0)$ and $y=(a,b)$, where $b\ge0$. The mapping $\psi$ rotates the circle $C\cap D$ clockwise by angle $\alpha$ (see Figure \ref{fig:rotate by alpha}), so that $y$ moves to the point $(r_1,0)$, and then radially projects the circle onto $B$. For points on any intermediate circle of radius $r$ ($r_0\le r\le r_1$), we do the rotation by the angle $\frac{r-r_0}{r_1-r_0}\alpha$ and then project the circle onto $B$. Now it is easy to see if $h=r_1-r_0$ is sufficiently large ($h\ge b$ suffices), the mapping $\psi$ defined this way is 1-Lipschitz. The details are left to the reader.
\end{proof}

\begin{theorem}
 Suppose that $X=X(G,w)$ is a metric graph that can be topologically 2-cell embedded into a surface $S$ and that $c(G)\ge3$.
 Then there is a polyhedral metric surface $\hat S$ homeomorphic to $S$, such that $c(\hat S) = c(X)$ and $c_0(\hat S) = c_0(X)$.
\end{theorem}

\begin{proof}
  The proof uses the following construction. Let $D$ be the diameter of the metric graph $X(G,w)$ and let $p$ be the maximum length of a facial walk in the 2-cell embedding of $G$ in $S$. By the length we mean the sum of the lengths of the edges participating in the facial walk. Set $h = D + p$.

  Now take the metric graph $X(G)$ and cap each facial walk $B$ with an expanding capped cylinder over $B$ of height $h$. This yields a metric surface $\hat S$ homeomorphic to $S$ in which $X$ is an isometric subset. Note that $B$ is not necessarily a simple closed curve, but Lemma \ref{lem:cylinder1expanding} still applies because in that lemma we have considered distances in the capped cylinder only and have assumed that the distances in $B$ are along the curve. This was in fact used by both, the cops and the robber, when they have mimicked the strategy from $X$.
\end{proof}

\begin{corollary}\label{cor:lower bound sqrt}
For every $g\ge0$ (and every $k\ge0$) there is an orientable (non-orientable) geodesic surface $X$ of genus $g$ (and with $k$ boundary components) such that $c(X)\ge g^{1/2-o(1)}$, where the asymptotics of $o(1)$ is considered for $g\to\infty$.
\end{corollary}

\begin{proof}
  The proof will only give a surface of genus at most $g$, but it is not hard to change it to genus exactly $g$ (which is left to the reader).

  Hosseini et al.~\cite{GoHoMo21} proved that for every sufficiently large $n$, there is a subcubic\footnote{A graph is \emph{subcubic} if all vertices have degree at most 3.} graph $G_n$ of order $n$, whose cop number is $c(G_n) = \Theta(n^{1/2-o(1)})$. Since a subcubic graph of order $n$ has at most $3n/2$ edges, it can be 2-cell embedded in a surface (orientable or nonorientable) of genus $g' \le n/2$. Lemma \ref{lem:cylinder1expanding} shows that there is a capped surface of genus $g'$ whose cop number is $\Theta(n^{1/2-o(1)})$. This completes the proof.
\end{proof}

\subsection{Upper bounds}

Concerning graphs that are embeddable in a surface of genus $g$, it was proved by Schroeder \cite{Schr01} that their cop number is at most $\tfrac{3}{2}g+3$. He conjectured that the constant $\tfrac{3}{2}$ can be replaced by $1$. The currently best bound towards the Schroeder Conjecture was obtained recently by Bowler, Erde, Lehner, and Pitz \cite{BoErLePi21}, who proved that Schroeder's bound can be improved to $\tfrac{4}{3}g+3$. Another, asymptotic improvement has been announced in an extended abstract at EuroComb 2019 \cite{BoErLePi19}.

We will not try to optimize the genus bounds as precisely as in the above-mentioned works since the author believes that the linear bound given in our next theorem is far from best possible (asymptotically). But we give a simple linear upper bound.

\begin{theorem}\label{thm:upper bound linear}
  Let $X$ be a compact geodesic surface of genus $g\ge1$ (homeomorphic to $S_{g,k}$ or $N_{g,k}$ for some $k\ge0$). Then $c(X) \le 2g+1$.
\end{theorem}

\begin{proof}
  If $g=1$, we either have the torus $S_{1,k}$ or the projective plane $N_{1,k}$. The bound for this case has been proved earlier (see Corollary \ref{cor:genus0and1}).

  If $g>1$, we take an essential systole of the capped surface, which is a geodesic noncontractible simple closed curve in $X$ and can be written as the union of two isometric paths in $X$. By using two cops to guard the systole, we essentially cut the surface into a geodesic surface of smaller genus, and the result follows by induction.
\end{proof}

The natural question is whether the linear upper bound of Theorem \ref{thm:upper bound linear} or the $\sqrt{g}$ lower bound of Corollary \ref{cor:lower bound sqrt} gives the right asymptotics. In parallel to a conjecture of the author about graphs of genus $g$, we propose the following generalized conjecture.

\begin{conjecture}
  Let\/ $X$ be a geodesic surface of genus $g$. Then $c(X) = O(\sqrt{g})$.
\end{conjecture}

\section{Where to go next?}

The main question that arises from our work is which geometric properties force the cop number to be large. Here we propose a very basic conjecture.

\begin{conjecture}
  Suppose that $X$ is an $n$-dimensional simplicial {pseudomanifold\/}\footnote{By a simplicial pseudomanifold we mean a simplicial complex in which each simplex is contained in an $n$-simplex, and each $(n-1)$-simplex is contained in at most two $n$-simplices.}, whose homology group $H_i(X)$ has rank $r_i$ for $i=1,\dots,n$. Then
  $c(X) = O(n\sqrt{r_1+\cdots+r_n}\,)$.
\end{conjecture}

It may be that $c_0(X)$ would fall within the same bound as conjectured above. In fact, in the examples that we understand, the factor $n$ of the conjecture seems to be necessary only for $c_0(X)$.

\bibliographystyle{plain}
\bibliography{references_CRsurfaces}

\appendix

\section{Making the boundary piecewise geodesic}

In order to use Theorem \ref{thm:triangulate surface epsilon}, we need piecewise geodesic boundary. In this appendix we prove that any more complicated boundary component can be well approximated by cutting an $\varepsilon$-neighborhood around it.

\begin{theorem}\label{thm:make boundary piecewise geodesic}
  Let $X$ be a compact geodesic surface and $\varepsilon>0$. Then $X$ contains a geodesic surface $X'\subseteq X$ that is homeomorphic to $X$ and has the following properties:
  \begin{itemize}
    \item[(i)] $X'$ is isometric in $X$.
    \item[(ii)] $X'$ has piecewise geodesic boundary.
    \item[(iii)] Every point in $X\setminus X'$ is at distance less than $\varepsilon$ from $X'$.
  \end{itemize}
\end{theorem}

\begin{proof}
  We may assume that the length of every noncontractible closed curve in $X$ and the minimum distance between two boundary components is at least $3\varepsilon$.
  This will enable us to treat each boundary component separately as all points in $X\setminus X'$ will be at distance at most $\varepsilon$ from the boundary.

  Let $B$ be a boundary component of $X$. We may assume that any closed curve $\gamma$ homotopic to $B$ with each of its points at distance less than $\varepsilon$ from $B$ has length $\ell(\gamma) > 4\varepsilon$.
We claim that there is a positive constant $r$, $0<r<\varepsilon/4$ such that for any $x,y\in B$ with $d(x,y)<r$ the subset of $X$ bounded by the $(x,y)$-segment of $B$ and any $(x,y)$-geodesic $\alpha$ has diameter at most $\varepsilon$. If not, we consider a sequence of pairs $(x_n,y_n)$ with $d(x_n,y_n)<1/n$ contradicting the stated property. By compactness, we may assume that $x_1,x_2,\dots$ converge to a point $x\in B$. Then also $y_1,y_2,\dots$ converge to $x$, and there is a sequence of corresponding geodesics $\alpha_1, \alpha_2,\dots$, each of which cuts a subset with a point $z_i$ at distance more than $\varepsilon/2$ from $x$. Let $z$ be a limit point of the sequence $z_1,z_2,\dots$. Then a subsequence of $(\alpha_i)_{i\ge1}$ cuts out the limit point $z$. That subsequence of geodesics converges to the point $x$, which means that $x$ is a ``pinch point'', contradicting the fact that $X$ is a surface. This proves the claim.

  In what follows we use the following notation. If $x,y\in B$, then we denote by $B[x,y]$ the $(x,y)$-segment on $B$, taken in the preferred direction on $B$. In our figures the surface will be on the left side of $B$ when we traverse $B$ in the preferred direction. If $\alpha:[0,1]\to X$ is a simple curve and points $x=\alpha(t)$ and $y=\alpha(t')$ are on $\alpha$ with $t\le t'$, then we denote by $\alpha[x,y]$ the segment of $\alpha$ from $x$ to $y$.

  Since $B$ is homeomorphic to a circle, for each point $x\in B$, there is a positive constant $0<r_x<r/2$ such that each point on an open interval $I_x$ around $x$ on $B$ is at distance in $X$ at most $r_x$ from $x$.
  Let $x_1,x_2,\dots,x_s$ be a finite set of points in $B$ of minimum possible cardinality such that $\{I_{x_i}\mid i\in[s]\}$ is a cover of $B$. (The fact that the set is finite follows from compactness.) We may assume that the points are listed in the cyclic order as they appear on $B$. We will use the notation $x_1<x_2<\cdots <x_s<x_1$ to denote this cyclic order and, in particular, imply that no other point $x_j$ is in the segment of $B$ from $x_i$ to $x_{i+1}$ ($i\in [s]$, $j\neq i, i+1$).

  Let us now consider two consecutive points $x_i$ and $x_{i+1}$ and their intervals $I_{x_i}$ and $I_{x_{i+1}}$. By the minimality of $s$, the union of these two (closed) intervals contains the whole $(x_i,x_{i+1})$-segment on $B$, so there is a point in the intersection, and hence $d(x_i,x_{i+1})\le r_{x_i} + r_{x_{i+1}} < r$. Thus every $(x_i,x_{i+1})$-geodesic cuts off a subset $A_i$ of diameter at most $\varepsilon$. By compactness, there is an $(x_i,x_{i+1})$-geodesic $\alpha_i$, for which $A_i$ is (inclusion-wise) minimal.
  Let $w_i$ ($z_i$) be the ``leftmost'' (``rightmost'') point of $\alpha_i[x_i,x_{i+1}]\cap B$, so that the order on $B$ is $w_i\le x_i < x_{i+1}\le z_i$ and $\alpha_i$ has no points in $B\setminus B[w_i,z_i]$. We also take the $(w_i,z_i)$-geodesic $\gamma_i=\alpha_i[w_i,z_i]$. The geodesic $\gamma_i$ together with the $(w_i,z_i)$-segment on $B$ bounds a \emph{degenerate disk}, which we denote by $D_i$.

\begin{figure}
  \centering
  \includegraphics[width=13.6cm]{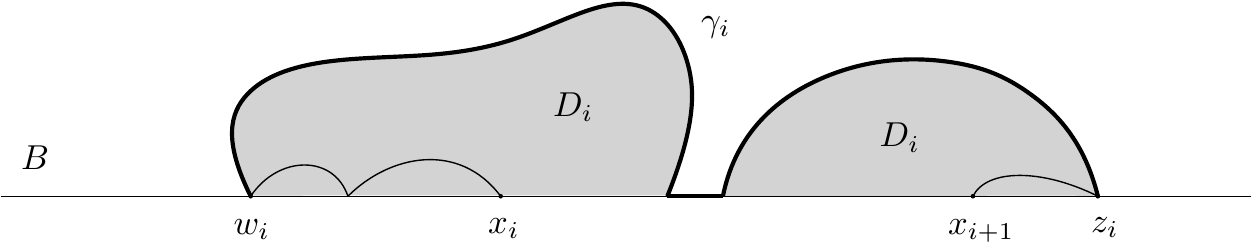}
  \caption{The figure shows the $(x_i,x_{i+1})$-geodesic $\alpha_i$ which may intersect $B$ outside of the $(x_i,x_{i+1})$-segment on $B$. Then we define the leftmost and the righmost points $w_i,z_i$, the geodesic $\gamma_i=\alpha_i[w_i,z_i]$ joining them, and the (degenerate) disk $D_i$ bounded by $\gamma_i$ and $B$. The $(w_i,z_i)$-geodesic $\gamma_i$ is shown bold; $D_i$ consists of the two shaded disks together with the joining segment on $B$.}\label{fig:define_wi_and_zi}
\end{figure}

The \emph{constituents} $w_j,z_j,\gamma_j,D_j$ for $j\in [s]$ have the following properties for every $j\in [s]$ (with all indices considered modulo $s$):

\begin{itemize}
    \item[(1)] The pairs $(w_j,z_j)$ and $(w_{j+1},z_{j+1})$ interlace on $B$, i.e.\ $w_j < w_{j+1} \le z_j < z_{j+1}$, and the union of all segments $B[w_j,z_j]$ covers $B$.
    \item[(2)] The $(w_j,z_j)$-geodesic $\gamma_j$ bounds a (possibly degenerate) disk $D_j$ together with $B[w_j,z_j]$, $\gamma_j$ is the unique geodesic from $w_j$ to $z_j$ contained in $D_j$, and $\gamma_j[w_j,z_j]\cap B \subseteq B[w_j,z_j]$.
    \item[(3)] The length of $\gamma_j$ is less than $r$, $\ell(\gamma_j)<r$, and the diameter of $D_j$ is smaller than $\varepsilon$.
\end{itemize}
We will change the constituents (and possibly decrease their number $s$) so that they will still satisfy (1)--(3), and will also satisfy the following:
\begin{itemize}
    \item[(4)] $\gamma_j$ intersects $\gamma_{j-1}$ and $\gamma_{j+1}$, but is disjoint from all other $\gamma_m$, $m\notin \{j-1,j,j+1\}$.
    \item[(5)] Let $x'_j$ be the first point on $\gamma_{j-1}$ that belongs to $\gamma_{j-1}\cap \gamma_j$, when $\gamma_{j-1}$ is traversed from $w_{j-1}$ towards $z_{j-1}$. Then the union $\cup_{j=1}^s \gamma_j[x'_j,x'_{j+1}]$ forms a simple closed curve $B'$ in $X$ that is homotopic to $B$.
\end{itemize}
Let us observe that (5) is a consequence of (1)--(4), so it remains to see how to achieve (4).

We take a set of constituents such that $s$ is smallest possible and properties (1)--(3) hold. Then it is clear that for any distinct $i,j\in[s]$, we cannot have $w_i\le w_j<z_j\le z_i$. In such a case we could remove the $j$th constituent. This implies that any two constituents either (weakly) interlace, i.e.\ $w_i<w_j\le z_i<z_j$ (in which case $B[w_i,z_i]\cap B[w_j,z_j]\not=\emptyset$), or they cover disjoint segments on $B$, $B[w_i,z_i]\cap B[w_j,z_j]=\emptyset$. Clearly, each constituent weakly interlaces with the previous one and the next one. But if it interlaces with another one, then three of them would cover the same point on $B$, and it is easy to see that we could remove one of them and still have properties (1)--(3). Thus, to show (4), it suffices to consider the possibility that $w_i<z_i < w_j<z_j$, and $\gamma_i$ and $\gamma_j$ intersect. Let $x$ be the first intersection point when we traverse $\gamma_i$ from $w_i$ towards $z_i$. Note that one of the segments $\gamma_i[w_i,x]$ and $\gamma_i[x,z_i]$ combined with one of $\gamma_j[w_j,x]$ and $\gamma_j[x,z_j]$ has length less than $r$ (since both $\gamma_i$ and $\gamma_j$ have length less than $r$). That combined curve thus bounds the disk together with the corresponding segment on $B$, an we may assume that the segment on $B$ contains $B[z_i,w_j]$.

If $\ell(\gamma_i[w_i,x]\cup \gamma_j[x,z_j])<r$, then we can replace the $i$th and $j$th constituents with one constituent using the points $w_i$ and $z_j$, thus decreasing $s$. Otherwise, if $\ell(\gamma_i[w_i,x]) + \ell(\gamma_j[w_j,x])<r$. In this case we replace $(w_i,z_i)$ with the pair $(w_i,w_j)$ (and make the corresponding constituent). Now we can eliminate the constituents covering $B[z_i,w_j]$. The remaining case when $\ell(\gamma_i[x,z_i]) + \ell(\gamma_j[x,z_j])<r$ is similar. This shows that (4) will hold.

Finally, the proof of (5) is easy (using the fact that any curve surrounding $B$ in its $\varepsilon$-neighborhood has length more than $4r$.

Now, we remove all points that are contained in the ``degenerate cylinder'' between $B$ and $B'$ and thus replace the boundary component $B$ with the piecewise geodesic boundary $B'$. Since all removed points are contained in disks using geodesic curves of length less than $r$, these disks have diameter at most $\varepsilon$. Also, each such disk has a point in $B'$. Therefore, each point in $X\setminus X'$ is at distance at most $\varepsilon$ from $X'$. This completes the proof.
\end{proof}

\end{document}